\documentclass[12pt,oneside,reqno]{amsart}
\usepackage{amsmath}
\usepackage{amssymb}
\usepackage{graphicx}
\usepackage{tikz-cd}

\usepackage{hyperref}

\usepackage[all]{xy}
\usepackage{xypic}
\usepackage{mathtools}

\usepackage{xspace}
\usepackage{bm}
\usepackage{amsmath}
\usepackage{amstext}
\usepackage{amsfonts}
\usepackage[mathscr]{euscript}
\usepackage{amscd}
\usepackage{latexsym}
\usepackage{amssymb}
\usepackage{enumerate}
\usepackage{xcolor}

\theoremstyle{plain}
\swapnumbers
    \newtheorem{theorem}{Theorem}[section]
    \newtheorem{proposition}[theorem]{Proposition}
    \newtheorem{lemma}[theorem]{Lemma}

    \newtheorem{subsec}[theorem]{}

\theoremstyle{definition}
    \newtheorem{definition}[theorem]{Definition}

\theoremstyle{remark}
        \newtheorem{remark}[theorem]{Remark}

\newcommand{\sect}{\setcounter{figure}{0}\section}

\renewcommand{\thefigure}{\arabic{section}.\arabic{figure}}

\newenvironment{myeq}[1][]
{\stepcounter{figure}\begin{equation}\tag{\thefigure}{#1}}
{\end{equation}}

\newcommand{\hdim}{{\sf {hdim}}}

\newcommand{\secat}{{\sf {secat}}}

\newcommand{\TC}{{\sf {TC}}}

\newcommand{\rr}{{\mathbb {R}}}
\newcommand{\br}{\bar{\mathbf{r}}}
\newcommand{\EIBr}{{E^{I_{\bar{\textbf{r}}}}_B}}
\newcommand{\EIB}{{E^{I_{(2,3)}}_B}}

\allowdisplaybreaks

\title[Parametrized Topological Complexity for Multi-Robot System]{Parametrized Topological Complexity for\\a Multi-Robot System with Variable Tasks}
\author{Gopal Chandra Dutta}
\address{Stat-Math Unit, Indian Statistical Institute, Kolkata 700108, India}
\email{gdutta670@gmail.com}
	
\author{Amit Kumar Paul}	
\address{Stat-Math Unit, Indian Statistical Institute, Kolkata 700108, India}
\email{amitkrpaul23@gmail.com}
	
\author{Subhankar Sau}
\address{Department of Mathematics and Statistics, IIT Kanpur, Kanpur 208016, India}
\email{subhankarsau18@gmail.com}
	
\thanks{ }

\date{\today}                                          

\begin{document}

\begin{abstract}
We study a generalized motion planning problem involving multiple autonomous robots navigating in a $d$-dimensional Euclidean space in the presence of a set of obstacles whose positions are unknown a priori. Each robot is required to visit sequentially a prescribed set of target states, with the number of targets varying between robots. This heterogeneous setting generalizes the framework considered in the prior works on sequential parametrized topological complexity \cite{FarP22, FarP23}. To determine the topological complexity of our problem, we formulate it mathematically by constructing an appropriate fibration.  
Our main contribution is the determination of this invariant in the generalized setting, which captures the minimal algorithmic instability required for designing collision-free motion planning algorithms under parameter-dependent constraints. We provide a detailed analysis for both odd and even-dimensional ambient spaces, including the essential cohomological computations and explicit constructions of corresponding motion planning algorithms.
\end{abstract}

\maketitle
\tableofcontents

\sect{Introduction}
The development of algorithmic methods for robot motion planning has emerged as an active and rapidly evolving area of research in robotics; we refer the reader to the monographs \cite{Lat91, LaV06} for further references. A significant contribution to this field was made by M. Farber, who introduced the notion of topological complexity $\TC(X)$ for a path-connected topological space $X$ to approach the motion planning problem in \cite{Far03}. The invariant $\TC(X)$ serves as a quantitative measure of the inherent discontinuity or instability present in any motion planning algorithm operating in the configuration space $X$. More precisely, $\TC(X)$ encodes the minimal number of continuous motion planning rules required to define a global algorithm in $X$.
The topological complexity of collision-free motion of multiple robots in an Euclidean space is studied in \cite{FarY04, FarG09}. Subsequently, in \cite{FarGY07}, the authors investigated the problem in the presence of multiple obstacles.

The concept of topological complexity was further generalized by Rudyak, who introduced the notion of $r$-th sequential topological complexity $\TC_r(X)$ for $r \in \mathbb{N}$ and $r \geq 2$, as seen in \cite{Rud10}. This extension captures the complexity of motion planning tasks in which a system must traverse a sequence of $r$ many states. In particular, the classical topological complexity is recovered as the special case $\TC_2(X) = \TC(X)$. The sequential variant thus broadens the scope of topological methods in motion planning by accommodating more complex task specifications involving multiple sequential goals.    

A parametrized framework for motion planning algorithms was introduced in \cite{CohFW21}, offering a more universal and adaptable approach to motion planning. Parametrized algorithms are designed to operate effectively under varying external conditions, which are incorporated into the input of algorithm as parameters. This framework is particularly relevant in scenarios involving the collision-free navigation of multiple objects (or robots) within a $d$-dimensional Euclidean space, where the positions of obstacles are not known in advance. From a mathematical perspective, the minimal degree of instability in such a motion planning problem is captured by the parametrized topological complexity of the Fadell–Neuwirth fibration \cite{CohFW21, CohFW22, FarW}.
Building upon this foundation, the authors in \cite{FarP22, FarP23} introduced and developed the theory of $r$-th sequential parametrized topological complexity. This generalization provides the measure of instability in motion planning tasks where each robot is required to sequentially reach the $r$ specified configurations in $\mathbb{R}^d$, while avoiding collisions with both obstacles and other robots. The authors provided a comprehensive analysis of this invariant in the context of the Fadell–Neuwirth fibration. 
  
In this article, we investigate a generalized and practically motivated scenario in the context of motion planning for multiple autonomous objects. We consider a mechanical system consisting of $n$ autonomous robots, denoted  $z_1, z_2, \dots, z_n$, operating in a $d$-dimensional Euclidean space $\rr^d$ in the presence of a set of obstacles with unknown a priori positions. 
The goal is to design a collision-free motion plan such that each robot $z_i$ sequentially visits $r_i$ prescribed target states in the space, while avoiding collisions with both the obstacles and the other robots. Crucially, the number of target states 
$r_i$ may vary with $i$ introducing heterogeneous task requirements across the robot ensemble. This characteristic marks a significant generalization of the setting studied in \cite{FarP22, FarP23}, where it was assumed that all robots follow the same number of sequential goals.

Our objective is to compute the parametrized topological complexity associated with this generalized scenario. This topological invariant measures the minimal degree of discontinuity or algorithmic instability required to solve the motion planning problem under a parameter-dependent condition. Our results extend the existing framework of sequential parametrized topological complexity and provide new insights into the structural complexity of realistic, heterogeneous multi-robot systems.

This article is organized as follows. 

We review the necessary background on topological complexity, parametrized topological complexity, and their sequential versions in Section~\ref{sec:preliminaries}. In Section \ref{sec:problem}, we examine a concrete example to illustrate the key ideas of our study. Specifically, we consider the problem of collision-free motion planning for two robots operating in three-dimensional Euclidean space in the presence of two obstacles whose positions are not known a priori. The first robot takes two stops, while the second robot takes three stops. We analyze this scenario in full detail and compute the corresponding parametrized topological complexity. This example motivates and clarifies the broader theoretical framework developed in the subsequent sections.

We then formally introduce the general version of the motion planning problem in Section \ref{sec: variable target} and derive an upper bound for its parametrized topological complexity in Proposition \ref{Prop_ubptc}.
To establish a lower bound, we study the associated cohomology algebra in Section \ref{sec:cohomology}.
Our analysis to find the lower bound using cup length in cohomology distinguishes between two cases based on the dimension of the ambient Euclidean space: the odd-dimensional case and the even-dimensional case.
For both cases, we compute the lower bound explicitly; see Proposition \ref{Prop_odd lower bound} for odd case and Proposition \ref{proposition lower bound even case} for even. In the odd-dimensional case, we find that the lower bound coincides with the general upper bound, thus determining the exact value of the parametrized topological complexity in Theorem \ref{Th_odd case}. 
In contrast, for the even-dimensional case, the lower and upper bounds do not agree. To resolve this discrepancy, we construct an explicit motion planning algorithm tailored to the even-dimensional setting in Section \ref{sec:algorithm}, which enables us to compute the exact value of the parametrized complexity in this case as well; see Theorem \ref{Th_even case}.
In Sections \ref{sec:odd} and \ref{sec:even}, while studying the cup length in both cases, we skip some of the computations to maintain the flow for the reader. We include a brief outline of those calculations in the Appendix (Section \ref{sec:appendix}).

\sect{Preliminaries}\label{sec:preliminaries}

In this section, we recall some notions related to sectional category, topological complexity, sequential topological complexity, configuration spaces, and sequential parametrized topological complexity. We also revisit some relevant results concerning these notions.

The \emph{sectional category} of a Hurewicz fibration $p: E\to B$, denoted by $\secat(p)$, is defined as the least non-negative integer $k$ such that there exists an open cover $\{U_0, \dots, U_k\}$ of $B$, where each open set $U_i$ admits a continuous section $s_i \colon U_i \to E$ of $p$. If no such $k$ exists, we set $\secat(p)=\infty$.

Let $X$ be a path-connected space and $X^I$ be the free path space, that is, the space of all continuous maps $I=[0,1] \to X$ equipped with compact-open topology. 
Consider the fibration $$\pi \colon X^I \to X \times X, \quad \text{ defined by } \pi(\alpha)=(\alpha(0), \alpha(1)).$$ 
The \emph{topological complexity} $\TC(X)$ of $X$ is defined as $\TC(X) := \secat(\pi)$. Topological complexity provides a quantitative measure of the inherent discontinuity or instability present in any motion planning algorithm. In general, computing this invariant is challenging. In most cases, rather than determining the exact value of $\TC(X)$, one focuses on finding upper and lower bounds, as various tools are available in the literature for this purpose. The upper bound is usually obtained from the dimensional arguments. On the other hand, to find a lower bound for the topological complexity, we use a classical result of Schwarz \cite{Sva66}, stated as follows.

\begin{lemma}\label{Lem_lower bound of secat}	
Let $p: E \to B$ be a fibration, and $R$ be a coefficient ring. If there exist cohomology classes $u_1, \cdots, u_k\in \ker[p^*: H^*(B; R)\to H^*(E; R)]$ such that their cup-product is nonzero, $u_1 \smile \cdots \smile u_k \neq 0 \in H^*(B; R)$, then $\secat(p) \geq k$.
\end{lemma}

For a fix $r \geq 2$, consider the evaluation map $\pi_r \colon X^I \to X^r$ defined as $$\pi_r(\alpha)=(\alpha(t_1), \dots, \alpha(t_r)),$$ where $0 \leq t_1 <t_2 < \dots < t_r \leq 1$. Often we refer to the points $t_1, \dots, t_r$ as \emph{time schedule}.
We recall that the definition of sequential topological complexity follows from \cite{Rud10}.
\begin{definition}
The \emph{$r$-th sequential topological complexity} of a path-connected space $X$ is denoted by $\TC_r(X)$, defined as $\TC_r(X):=\secat (\pi_r)$. 
\end{definition}
It is clear that $\TC_2(X)=\TC(X)$.
Let $Y$ be a path-connected topological space. The configuration space of $n$ distinct ordered points lying in $Y$, denoted by $F(Y, n)$, is defined as follows:
\begin{equation*}
F(Y, n)= \{(y_1, \dots, y_n)\in Y^n~ \colon ~ y_i\neq y_j \text{ if } i\neq j\}.
\end{equation*}

For $r,d,n\geq 2$, the $r$-th sequential topological complexity of the configuration space $F(\rr^d, n)$ 
is given by
\begin{equation}\label{sq tc of con sp}
\TC_r(F(\rr^d, n)) = 
\begin{cases}
r(n-1) & \mbox{ if } ~~ d \mbox{ is odd },\\
r(n-1)-1 &\mbox{ if } ~~ d \mbox{ is even },
\end{cases}
\end{equation}
(see \cite[Theorem 4.1]{GonG15}).
The result \eqref{sq tc of con sp} describes the sequential topological complexity of the practical problem: the collision-free motion of $n$ robots, each required to pass through $r$-many target points in $d$-dimensional Euclidean space.

Cohen, Farber, and Weinberger introduced the notion of parametrized topological complexity for a Hurewicz fibration in \cite{CohFW21}. Subsequently, in \cite{FarP22}, Farber and Paul generalized this concept to the sequential parametrized topological complexity. We consider a Hurewicz fibration $p: E \to B$ and obtain the space
\begin{equation}\label{Eq_ErB}
E^r_B= \{(e_1, \cdots, e_r)\in E^r: \, p(e_1)=\cdots = p(e_r)\}.
\end{equation}
Let $E^I_B$ be the space of all paths in $E$ that lie in a single fibre of $E$ under the map $p\colon E\to B$, i.e. 
\begin{equation*}
 E^I_B=\{\alpha: I = [0,1]\to E \text{ such that } p\circ \alpha \text{ is constant path in } B\}.
\end{equation*}
For a fix $r \ge 2$, consider the evaluation map 
\begin{equation*}
\Pi_r : E^I_B \to E^r_B, \text{ defined as } \quad \Pi_r(\alpha) = (\alpha(t_1), \alpha(t_2), \dots,  \alpha(t_r)),
\end{equation*} 
where $0\leq t_1\leq \ldots \leq t_r\leq 1$. Note that the section of the fibration $\Pi_r$ can be realized as an algorithm for sequential parametrized motion. 

\begin{definition} \label{definition sequential parametrized topological complexity}
The \emph{$r$-th sequential parametrized topological complexity} of the given fibration $p : E \to B$ is denoted by $\TC_r[p : E \to B]$ and is defined as 
$$\TC_r[p : E \to B]:=\secat(\Pi_r).$$
\end{definition}

Consider the map $p: F(\rr^d, m+n)\to F(\rr^d, m)$ defined as $$p(x_1, x_2,\ldots, x_{m+n}) = (x_1,x_2,\ldots,x_m).$$
Then $p$ is a locally trivial fibration (see \cite{FadN62}). It is known as the \emph{Fadell-Neuwirth fibration}. Let $n\geq 1$ and $r,d,m\geq 2$. The $r$-sequential parametrized topological complexity of Fadell-Neuwirth fibration is given by
\begin{equation}\label{paratc of con sp}
\TC_r[p: F(\rr^d, m+n) \to F(\rr^d, m)] = 
\begin{cases}
rn+m-1 & \mbox{ if } ~~ d \mbox{ is odd },\\
rn+m-2 &\mbox{ if } ~~ d \mbox{ is even },
\end{cases}
\end{equation}
(see \cite{FarP22} and \cite{FarP23}).
The result in \eqref{paratc of con sp} describes sequential parametrized topological complexity for a practical scenario: the collision-free motion of $n$ robots, each required to pass through $r$-many target points, in the presence of $m$-many obstacles in a $d$-dimensional Euclidean space with unknown a priori positions. 

In this paper, we focus on a generalized version of this problem, where each robot is required to attain a possibly different number of target points than the others.

\sect{A motivating problem}\label{sec:problem}

In this section, we discuss a motion planning problem in $\mathbb{R}^3$ where two robots, say $z_1$ and $z_2$, move in the presence of two obstacles $o_1$ and $o_2$, without any collision. The first robot $z_1$ takes two stops (initial stop and final stop), while the second robot $z_2$ takes three stops (initial and final stop along with an intermediate stop). Let us formulate this problem mathematically. 

We consider the following Fadell-Neuwirth fibration
\begin{equation}\label{Eq_fibsec example}
p: E=F(\rr^3, 4) \to B=F(\rr^3, 2) \text{ defined as } (o_1, o_2, z_1, z_2)\mapsto (o_1, o_2).
\end{equation}
As discussed in \eqref{Eq_ErB}, we have the space $$E^3_B=\{(e_1, e_2, e_3)\in E^3 ~ \colon  ~p(e_1)=p(e_2)=p(e_3)\}.$$ 
We consider another fibration \begin{equation}\label{eq_p1_example}
 p_1: E=F(\rr^3, 4) \to F(\rr^3, 3) \text{ defined as } (o_1, o_2, z_1, z_2)\mapsto (o_1, o_2, z_1).
 \end{equation}
We define the subspace $E_B^{(2, 3)}$ of $E^3_B$ as
\begin{equation}\label{Eq_E23B}
E_B^{(2, 3)}:=\{(e_1, e_2, e_3)\in E^3_B :  p_1(e_2)=p_1(e_3)\}.
\end{equation}

Notice that an element in $E_B^{(2, 3)}$ represents three configuration points $e_1, e_2, e_3$, where the configurations $e_s = (o_1, o_2, z^s_1, z^s_2)$ for $1\leq s\leq 3$ satisfying the following conditions:
\begin{enumerate}[(i)]
\item $o_1 \neq o_2$,
\item $o_j\neq z^s_i$ for $i,j=1,2$,
\item $z^s_1\neq z^s_2$,
\item $z^2_1 = z^3_1$.
\end{enumerate}
It follows that a point in $E^{(2,3)}_B$ can be viewed as $(o_1, o_2, z_1^1, z_1^2, z_2^1, z_2^2, z_2^3)\in (\rr^3)^7$ satisfying conditions (i), (ii), and (iii).

As in the discussion in Section \ref{sec:preliminaries}, it follows that an element $\alpha \in E^I_B$ can be written as $$\alpha(t) = (o_1, o_2,\alpha_1(t), \alpha_2(t)), ~t\in I.$$ Note that $o_1$ and $o_2$ are the positions of the obstacles associated to the path $\alpha$, satisfies the following conditions: 
\begin{flalign}\label{Eq_condition in paths}
\text{(i)}& ~o_1\neq o_2,\nonumber \\ 
\text{(ii)}& ~ \alpha_1(t)\neq \alpha_2(t) \text{ for all } t\in I, \\  
\text{(iii)}& ~ \alpha_i(t) \neq o_j \text{ for all } t\in I \text{ and } i,j=1,2.\nonumber
\end{flalign}
Therefore, we obatin the fibration $$\Pi_3: E^I_B \to E^3_B \quad \text{is defined by} \quad \alpha \mapsto (\alpha(0), \alpha(\frac{1}{2}), \alpha(1)).$$ 

Here, we consider a space $\EIB$
that fits into the following pullback diagram:
$$
\xymatrix{
\EIB \ar[rr]^{} \ar[dd]_{\Pi_{(2, 3)}} && E^I_B \ar[dd]^{\Pi_3} \\ \\
E_B^{(2, 3)} \ar@{^{(}->} [rr]_{\iota} && E^3_B
}
$$
The space $\EIB$ can be identified with the inverse image of $E_B^{(2, 3)}$ under the map $\Pi_3$ and $\Pi_{(2, 3)}: \EIB \to E_B^{(2, 3)}$ is the restriction map of $\Pi_3$.
So, an element $\alpha$ in $\EIB$ satisfies the conditions (i), (ii), (iii) in \eqref{Eq_condition in paths} along with an additional condition
$$\text{(iv)}~\alpha_1(\frac{1}{2})=\alpha_1(1).$$

The fibration $\Pi_{(2, 3)}$ can be express as $$\Pi_{(2, 3)}: \EIB \to E_B^{(2, 3)}, \quad \alpha \mapsto (o_1,o_2,\alpha_1(0), \alpha_1(\frac{1}{2}), \alpha_2(0), \alpha_2(\frac{1}{2}), \alpha_2(1)).$$
The sectional category $\secat(\Pi_{(2, 3)})$ is the topological complexity of this problem and we denote it by $\TC_{(2, 3)}[p: E \to B]$.
From the above pullback diagram it is clear that $$\TC_{(2, 3)}[p: E \to B] \leq \TC_3[p: E \to B].$$ 
In the following discussion, we will find the number $\TC_{(2, 3)}[p: E \to B]$.

\vspace{.25cm}
\noindent \textbf{Upper Bound:} 
The fibre $X$ of the fibration \eqref{Eq_fibsec example} is given by 
$$p^{-1}\{(o_1, o_2)\} \cong \{(z_1,z_2)\in \rr^3-\{o_1, o_2\} \times \rr^3-\{o_1, o_2\}: z_1\neq z_2 \} = F(\rr^3-\{o_1, o_2\},2).$$
Since the connectivity of $X$ is $1$, we obtain the following upper bound
\begin{align*}
\TC_{(2, 3)}[p: E \to B]=\secat(\Pi_{(2, 3)}) &<  \frac{\hdim (E_B^{(2, 3)})+1}{2} \\
&\leq \frac{ 2 \hdim (X) +\hdim (X_1) + \hdim (B) + 1}{2},
\end{align*}
where $X_1$ is the fibre of the fibration \eqref{eq_p1_example} and $\hdim$ denotes the homotopy dimension, defined as follows $$\hdim(Z)= \min \big\{\dim(Y): Y \text{ is homotopy equivalent to } Z \big\}.$$
The first inequality follows from \cite[Theorem 5]{Sva66}. For the second inequality, consider the fibration
\[
\hat{p} : E_B^{(2,3)} \to B, \quad (e_1, e_2, e_3) \mapsto p(e_1).
\]
Observe that we can fit the fibre \(\hat{p}^{-1}\{(o_1,o_2)\}\) into the following locally trivial fibration:
\[
X_1 \xrightarrow{\iota} \hat{p}^{-1}\{(o_1,o_2)\} \xrightarrow{\rho} X \times X,
\]
where \(\rho (e_1, e_2, e_3) = (e_1, e_2)\). It follows that $$\hdim(\hat{p}^{-1}\{(o_1,o_2)\}) \leq \hdim(X\times X) + \hdim(X_1) = 2\hdim(X)+ \hdim(X_1).$$
Therefore,
$$\hdim E_B^{(2, 3)} \leq \hdim (\hat{p}^{-1}\{(o_1,o_2)\}) + \hdim B \leq 2\hdim(X)+ \hdim(X_1) + \hdim(B).$$ The homotopy dimensions of $X, X_1$ and $B$ are $4, 2,$ and $2$ respectively. Thus, $$\TC_{(2, 3)}[p: E \to B]\leq 6.$$

\vspace{.25cm}
\noindent \textbf{Lower Bound:} In the remainder of this section, we compute the lower bound of the topological complexity using cup length. We also find that the lower bound coincides with the upper bound, which provides us with the exact value of $\TC_{(2, 3)}[p: E \to B]$.

\begin{proposition}\label{Prop_cup product example}
Let $p\colon E\to B$ be the fibration defined in \eqref{Eq_fibsec example} and $R$ be a coefficient ring. Consider the diagonal map $\Delta : E \to E^{(2, 3)}_B$ defined as $\Delta(e)= (e, e, e)$. For the homomorphism $\Delta^\ast: H^\ast(E^{(2, 3)}_B;R) \to H^\ast(E;R)$ induced by $\Delta$, if there exist cohomology classes $u_1, \ldots, u_k \in \ker[\Delta^* : H^\ast(E^{(2, 3)}_B;R) \to H^*(E; R)]$ such that $$u_1 \smile \cdots \smile u_k \neq 0 \in H^\ast(E^{(2, 3)}_B;R)$$ then $$\TC_{(2, 3)}[p: E \to B] \geq k.$$ 
\end{proposition}

\begin{proof}
Define a map $c : E \to E_B^I$ by $c(e)(t) = e$ for all $t\in I$, that is, the constant path. Notice that the map $c$ is a homotopy equivalence and satisfies the following commutative diagram: 
$$
\xymatrix{
	E \ar[rr]^{c} \ar[dr]_{\Delta}&&E_B^I \ar[dl]^{\Pi_{(2, 3)}}\\
	& E^{(2, 3)}_B & &
}
$$
It follows that the induced map $c^*\colon H^*(E^I_B;R)\to H^*(E;R)$ is an isomorphism. Hence $$\ker[\Pi_{(2, 3)}^*: H^\ast(E^{(2, 3)}_B;R) \to H^*(E_B^I; R)] = \ker[\Delta^*: H^\ast(E^{(2, 3)}_B;R) \to H^*(E; R)].$$ Now, using Lemma \ref{Lem_lower bound of secat}, we get $\TC_{(2, 3)}[p: E \to B]=\secat(\Pi_{(2, 3)})\geq k$. 
\end{proof}
 
Recall that the integral cohomology ring $H^\ast(F(\rr^3, 4))$, as described in \cite[Theorem V.4.2]{FadH01}, contains the cohomology classes $w_{12}, w_{13}, w_{14}, w_{23}, w_{24}, w_{34}$ of degree $2$. Note that for $1\leq i< j\leq 4,$ the cohomology class $w_{ij}$ is obtained as the pullback of the fundamental class, say $u$, under the map $\phi_{ij}\colon F(\rr^3,4)\to S^2$ defined as $(x_1,x_2,x_3,x_4)\mapsto \frac{x_i-x_j}{||x_i-x_j||}$, i.e., $w_{ij} = \phi_{ij}^*(u)$. It is clear that $w_{ij} = - w_{ji}$. Moreover, the cohomology classes satisfy the following relations: 
 \begin{equation}\label{Eq_relations in example}
 (w_{ij})^2=0 \quad \mbox{and}\quad w_{ip}w_{jp}= w_{ij}(w_{jp}-w_{ip})\quad \text{for all }\,  1\leq i<j<p\leq 4.
 \end{equation}

\begin{proposition}\label{Prop_cohom ring E23}
The integral cohomology ring $H^*(E_B^{(2,3)})$ contains cohomology classes $w^s_{ij}$ of degree $2$, where $1\leq i < j \leq 4$ and $1\leq s \leq 3$, satisfying the following relations:
\begin{enumerate}[(a)]
\item $(w^s_{ij})^2=0$
\item $w^s_{ip} w^s_{jp}= w^s_{ij} (w^s_{jp} - w^s_{ip}) \text{ for } i<j<p$,
\item $w^s_{12} = w^{s'}_{12} \text{ for } 1\leq s \leq s' \leq 3$,
\item $w^2_{13}=w^3_{13}$ and $w^2_{23}=w^3_{23}$.
\end{enumerate}
\end{proposition}

\begin{proof} 
For $1\leq s\leq 3$, consider the projection map $q_s: E_B^{(2, 3)} \to E$  defined as $$(e_1,e_2,e_3)\mapsto e_s, \text{ where } e_s = (o_1, o_2, z^s_1, z^s_2).$$ 
One can write an element $e_s = (o_1, o_2, z^s_1, z^s_2)$ in $E$ as $(x_1, \ldots, x_4)$ where $x_i = o_i$ for $1\leq i\leq 2$ and $x_i = z^s_{i-2}$ for $3\leq i \leq 4$. Therefore, the cohomology class $w_{ij}\in H^2(E)$ induces a cohomology class $w^s_{ij}$ in $H^2(E_B^{(2,3)})$, defined as $$w^s_{ij} := (q_{s})^*(w_{ij}).$$
Thus we get the desired results using \eqref{Eq_E23B} and \eqref{Eq_relations in example}.
\end{proof}

\noindent In the following Proposition, we find a lower bound of $\TC_{(2, 3)}[p: E \to B]$.

\begin{proposition}
For the fibration in \ref{Eq_fibsec example}, we have $\TC_{(2, 3)}[p: E \to B] \geq 6.$
\end{proposition}

\begin{proof}
Recall that the diagonal map $\Delta : E \to E^{(2, 3)}_B$ is defined by $\Delta(e)= (e, e, e)$. For $1\leq s,s'\leq 3$ and $1\leq i < j \leq 4$, we have $$\Delta^*(w^s_{ij}) = (q_s\circ \Delta)^*(w_{ij}) = (q_{s'}\circ \Delta)^*(w_{ij}) = \Delta^*(w^{s'}_{ij}).$$
Thus $w^s_{ij} - w^{s'}_{ij}\in \ker(\Delta^*)$ for $1\leq s,s'\leq 3$ and $1\leq i < j \leq 4$. 

It is sufficient to show that $(w^2_{13}-w^1_{13})^2(w^2_{14}-w^1_{14})^2(w^2_{23}-w^1_{23})(w^3_{14}-w^1_{14})\neq 0.$ Note that $(w^2_{13}-w^1_{13})^2 = -2 w^2_{13} w^1_{13}$
and $(w^2_{14}-w^1_{14})^2 = -2 w^2_{14} w^1_{14}$. Thus, it is equivalent to show $w^2_{13}w^1_{13}w^2_{14}w^1_{14}(w^2_{23}-w^1_{23})(w^3_{14}-w^1_{14})\neq 0$ which is computed in the following
\begin{align*}
&~w^2_{13}w^1_{13}w^2_{14}w^1_{14}(w^2_{23}-w^1_{23})(w^3_{14}-w^1_{14}) 
\\= &~ w^2_{13}w^1_{13}w^2_{14}w^1_{14}(w^2_{23}w^3_{14}-w^1_{23}w^3_{14}) \quad \quad \quad (\text{since~} (w^1_{14})^2=0)
\\= &~ w^2_{14}w^1_{14}w^3_{14} \{ (w^2_{13}w^2_{23})w^1_{13} - w^2_{13}(w^1_{13}w^1_{23}) \}
\\= &~ w^2_{14}w^1_{14}w^3_{14} \{ w^2_{12}(w^2_{23}-w^2_{13})w^1_{13}-w^2_{13}w^1_{12}(w^1_{23}-w^1_{13}) \} \quad (\text{using Proposition } \ref{Prop_cohom ring E23}(b))
\\= & ~ w^2_{14}w^1_{14}w^3_{14} w^2_{12}w^2_{23}w^1_{13}-
 w^2_{14}w^1_{14}w^3_{14} w^2_{12}w^2_{13}w^1_{13}-
  w^2_{14}w^1_{14}w^3_{14} w^2_{13}w^1_{12}w^1_{23}
   \\& \hspace{3cm} +w^2_{14}w^1_{14}w^3_{14} w^2_{13}w^1_{12}w^1_{13}
 \neq 0 \quad (\text{being basis elements}).
\end{align*} 
Then applying Proposition \ref{Prop_cup product example}, we get the desired result.
\end{proof}

\noindent In conclusion, the topological complexity of the above problem is given by $$\TC_{(2, 3)}[p: E \to B] = 6.$$

\sect{Motion planning of multiple robots for variable target states}\label{sec: variable target}

In this section, we first formulate the main problem of the paper: collision-free motion planning for $n$ robots, denoted $z_1, \dots, z_n$ visiting $r_1, \dots, r_n$ target points, respectively, in the presence of $m$ obstacles $o_1, \dots, o_m$ in $\mathbb{R}^d$, where $d \geq 2$. Then we derive a general upper bound for its sequential parametrized topological complexity. The inspection for lower bound using the cup length of associated cohomology algebra is discussed in subsequent sections.

To formulate the problem mathematically, we may assume $r_1 \leq r_2 \leq \cdots \leq r_n$. In fact, we can rearrange the values $r_1, \dots, r_n$ in non-decreasing order and rename the corresponding robots accordingly. We denote this tuple by $\br = (r_1,\dots , r_n)$. 
Consider the Fadell-Neuwirth fibration 
\begin{equation}\label{Eq_FNB}
p:E = F(\mathbb{R}^d, m+n) \to B = F(\mathbb{R}^d,m) 
\end{equation}
defined by $(o_1, \dots, o_m, z_1, \dots , z_n) \mapsto (o_1, \dots, o_m).$ Now onwards, by $E$ and $B$ we mean the configuration spaces $F(\mathbb{R}^d, m+n)$ and $F(\mathbb{R}^d,m) $, respectively.

Assume that there are $\ell$ many distinct numbers among  $r_1, \dots, r_n$, so that we have the following relation, $$r_1 = \cdots = r_{n_1}< r_{n_1+1} =  \cdots = r_{n_2}<\cdots < r_{n_{\ell-1}+1} = \cdots =r_{n_\ell}= r_{n}.$$ Note that $n_{\ell} = n$ and we fix $n_0= 0$.
For $u=1, 2, \ldots, \ell-1$, we consider the following fibrations:
\begin{equation}\label{Eq_pu}
p_{u}\colon E = F(\rr^d, m+n)\to F(\rr^d, m+n_u)
\end{equation}
defined by $(o_1, \dots, o_m, z_1, \dots, z_n) \mapsto (o_1, \dots, o_m, z_1, \dots, z_{n_u})$.
We define a subspace $E^{\br}_B$ of $E^{r_n}_B$ as follows: 
\begin{equation}\label{Eq_description of total space}
E^{\br}_B=\Big\{ (e_1,\ldots,e_{r_n})\in E^{r_n}_B ~:~ p_{u}(e_{r_{n_u}}) = p_{u}(e_{r_{n_u}+1})= \cdots = p_{u}(e_{r_n}),  ~ 1\leq u \leq \ell-1 \Big\}.
\end{equation}

An element in $E^{\br}_B$ represents $r_n$ many configuration points, say $e_1, \ldots, e_{r_n}$, where the configurations $e_s = (o_1, \ldots, o_m, z^s_1,\ldots, z^s_n)$ for $1\leq s\leq r_n$ satisfying the following relations:
\begin{enumerate}[(i)]
\item $o_j \neq o_{j'}$ for $1\leq j \neq j'\leq m$,
\item $o_j \neq z^s_i$ for $1\leq i\leq n$,
\item $z^s_i\neq z^s_{i'}$ for $1\leq i\neq i'\leq n$,
\item $z^s_{i} = z^{s'}_{i}$ for $r_{i}\leq s,s' \leq r_n, ~ 1\leq i\leq n$.
\end{enumerate}
Therefore, an element in $E^{\bar{\textbf{r}}}_B$ can be expressed as
\begin{equation*}
(o_1, \dots, o_m, z^1_{1}, \dots , z_1^{r_1}, \dots z^1_{n}, \dots , z_n^{r_n}) \in (\rr^d)^{R+m}, \quad {\text{where } R=\sum_{i=1}^n r_i,}
\end{equation*}
satisfying relations (i), (ii), and (iii).

Further, recall the fibration $\Pi_{r_n}: E^I_B \to E^{r_n}_B$ defined as $\alpha\mapsto \big(\alpha(t_1),\ldots, \alpha(t_{r_n})\big)$, where $0\leq t_1 < \cdots < t_{r_{n}}\leq 1$.
We consider a subspace $\EIBr$ of $E^I_B$ that fits in the following pullback diagram:
\begin{equation}\label{diagram pullback}
\xymatrix{
\EIBr \ar[rr]^{} \ar[dd]_{\Pi_{\br}} && E^I_B \ar[dd]^{\Pi_{r_n}} \\ \\
E_B^{\br} \ar@{^{(}->} [rr]_{\iota} && E^{r_n}_B
}
\end{equation}
Observe that an element $\alpha\in \EIBr$ can be written as: $$\alpha(t) = (o_1,\ldots,o_m,\alpha_1(t),\ldots,\alpha_n(t)), ~t\in I$$ and satisfies the following conditions:
\begin{enumerate}[(i)]
\item $o_j \neq o_{j'}$ for $1\leq j \neq j'\leq m$,
\item $\alpha_i(t) \neq \alpha_{i'}(t)$ for all $t\in I$ and $1\leq i\neq i'\leq n$,
\item $\alpha_i(t) \neq o_j$ for all $t\in I$, and $1\leq i\leq n,~ 1\leq j\leq m,$
\item $\alpha_i(t_{r_i}) = \alpha_i(t_{r_{i+1}}) = \cdots = \alpha_i(t_{r_n})$ for $1\leq i\leq n.$
\end{enumerate}
So, we can realize the fibration $\Pi_{\br}$ as:
\begin{equation}\label{Eq_pi r bar}
\Pi_{\br} : \EIBr \to E^{\br}_B, \quad \alpha \mapsto \Big(o_1, \dots, o_m, \alpha_1(t_1), \dots, \alpha_1(t_{r_1}), \dots, \alpha_n(t_1), \dots, \alpha_n(t_{r_n})\Big).
\end{equation}

\begin{definition}
The $\bar{\textbf{r}}$-th sequential parametrized topological complexity of the Fadell-Neuwirth fibration $p\colon E\to B$, denoted by $\TC_{\bar{\textbf{r}}}[p \colon E \to B]$, is defined as the sectional category of the fibration $\Pi_{\bar{\textbf{r}}}$, i.e.
$$\TC_{\bar{\textbf{r}}}[p \colon E \to B] := \secat(\Pi_{\bar{\textbf{r}}}).$$
\end{definition}

\noindent From the diagram (\ref{diagram pullback}), it is clear that $\TC_{\br}[p: E \to B] \leq \TC_{r_n}[p: E \to B]$. In the remainder of this section, we find an upper bound for $\TC_{\br}[p \colon E \to B]$.

\begin{proposition}\label{Prop_ubptc}
For the Fadell-Neuwirth fibration $p\colon E \to B$ as in \eqref{Eq_FNB}, we have $$\TC_{\bar{\textbf{r}}}[p \colon E \to B] \leq \displaystyle\sum^{n}_{i=1} {r_i} + m -1.$$
\end{proposition}

We prove the proposition using the following lemma. We denote the fibre of the fibration $p \colon E \to B$ in \eqref{Eq_FNB} by $X$. We also denote the fibre of the fibration $p_u\colon E\to F(\mathbb{R}^d, m+n_u)$ in \eqref{Eq_pu} by $X_u$ for $u=1, \dots, \ell-1$.

\begin{lemma}\label{Lemma_general upper bound of TC}
Let $p\colon E \to B$ be the Fadell-Neuwirth fibration \eqref{Eq_FNB} with fibre $X$. The $\bar{\textbf{r}}$-th sequential parametrized topological complexity
\begin{align*}
& \TC_{\bar{\textbf{r}}}[p\colon E \to B]< \frac{\hdim (E^{\bar{\textbf{r}}}_B) + 1}{d-1} \\
& \leq \frac{ r_{n_1} \hdim (X) + (r_{n_2}-r_{n_1}) \hdim (X_1) + \cdots + (r_{n_\ell}-r_{n_{\ell-1}}) \hdim (X_{\ell-1}) + \hdim (B) + 1}{d-1}.
\end{align*}
\end{lemma}

\begin{proof}
Recall 
\begin{align*}
X = p^{-1}\{(o_1, \dots, o_m)\} \cong F(\rr^d-\{o_1,\ldots,o_m\},n). 
\end{align*}
It follows that the connectivity of $X$ is $(d-2)$. Hence, using the definition of $\TC_{\bar{\textbf{r}}}[p\colon E \to B]$ and \cite[Theorem 5]{Sva66}, we get the first strict inequality. 
Further, consider a fibration 
\begin{equation}\label{Eq_hat p}
\hat{p}\colon E^{\bar{\textbf{r}}}_B\to B \quad \text{ defined as } (e_1, \dots, e_{r_n})\mapsto p(e_1).
\end{equation}
Observe that for the fibre $\hat{p}^{-1}\{(o_1,\ldots,o_m)\}$, we have $$\hdim(\hat{p}^{-1}\{(o_1,\ldots,o_m)\}) \leq \hdim(X^{r_{n_1}}) + \hdim \big(X^{(r_{n_2} - r_{n_1})}_1\big) + \cdots + \hdim \big( X^{(r_{n_\ell} - r_{n_{\ell - 1}})}_{\ell - 1}\big).$$ 
Thus, 
\begin{align*}
\hdim(E^{\bar{\textbf{r}}}_B) & \leq \hdim(\hat{p}^{-1}\{(o_1,\ldots,o_m)\}) +\hdim(B) \\
& \leq \hdim(X^{r_{n_1}}) + \hdim\big(X^{(r_{n_2} - r_{n_1})}_1\big) + \cdots + \hdim \big(X^{(r_{n_\ell} - r_{n_{\ell - 1}})}_{\ell - 1}\big) + \hdim(B)\\
& \leq r_{n_1} \hdim (X) + (r_{n_2} - r_{n_1}) \hdim (X_1)+ \cdots + (r_{n_\ell} - r_{n_{\ell-1}}) \hdim (X_{\ell-1}) + \hdim (B). 
\end{align*}
Hence, we get the desired result.
\end{proof}

\begin{proof}[Proof of Proposition \ref{Prop_ubptc}]
Note that $\hdim(X) = n(d-1)$, $\hdim(B) = (m-1)(d-1)$ and $\hdim(X_u) = (n-n_u)(d-1)$ for $u= 1, \dots, \ell-1$. Therefore,
\begin{align*}
&\TC_{\bar{\textbf{r}}}[p\colon E \to B] \\
& < \frac{(d-1)[n r_{n_1} + (n - n_1) (r_{n_2} - r_{n_1}) + \cdots + (n - n_{\ell-1}) (r_{n_\ell} - r_{n_{\ell-1}}) + (m-1)] + 1}{d-1}\\
& = n_1 r_{n_1} + (n_2 - n_1) r_{n_2} + \cdots + (n - n_{\ell-1}) r_{n_\ell} + m -1 + \frac{1}{d-1}\\
& = \displaystyle\sum^{n}_{i=1} {r_i} + m -1 + \frac{1}{d-1}.
\end{align*}
Since $d\geq2$, we have our result.
\end{proof}

\sect{Cohomology algebra of $H^{*}(E^{\br}_B)$}
\label{sec:cohomology}

In this section, we explore the cohomologies that further lead us to find the lower bound for $\TC_{\br}[p\colon E \to B]$ using the following proposition.

\begin{proposition}\label{Prop_lbptc}
Let $p\colon E \to B$ be the Fadell-Neuwirth fibration \eqref{Eq_FNB}. Consider the diagonal map $\Delta \colon E\to E^{\bar{\textbf{r}}}_B$ defined as $\Delta(e) = (e, \dots, e)$. If there exist cohomology classes $u_1, \dots, u_k\in \ker[\Delta^{\ast}\colon H^{\ast}(E^{\bar{\textbf{r}}}_B; R)\to H^{\ast}(E; R)]$ such that $u_1 \smile \cdots \smile u_k \neq 0$ in $H^{\ast}(E^{\bar{\textbf{r}}}_B; R)$, where $R$ is a coefficient ring. Then $$\TC_{\bar{\textbf{r}}}[p\colon E \to B] \geq k.$$
\end{proposition}

The proof of this proposition follows from an argument similar to that in the proof of Proposition \ref{Prop_cup product example}.
For $E = F(\rr^d, m+n)$, we recall the integral cohomology ring $H^\ast(E)$, as described in \cite[Chapter V, Theorem 4.2 and 4.3]{FadH01}.
\begin{lemma}\label{Lem_cohomology ring of E}
The integral cohomology ring $H^\ast(E)$ contains cohomology classes $w_{ij}$ of degree $(d-1)$ , where $1\leq i< j\leq m+n,$ which multiplicatively generate $H^*(E)$ and satisfy the following relations: $$(w_{ij})^2=0 \quad \mbox{and}\quad w_{ip}w_{jp}= w_{ij}(w_{jp}-w_{ip})\quad \text{for all }\,  i<j<p.$$
\end{lemma}

Remark that the above lemma is applicable for any configuration space of $\rr^d$.

\begin{proposition}\label{Prop_relations on cohomology classes}
The integral cohomology ring $H^*(E_B^{\br})$ contains cohomology classes $w^s_{ij}$ of degree $(d-1)$, where $1\leq s \leq r_n$ and $1\leq i < j \leq m+n$, satisfying the following relations: 
\begin{enumerate}[(a)]
\item $(w_{ij}^s)^2=0 \quad \text{for } 1\leq s \leq r_n \text{ and } 1\leq i < j \leq m+n$.
\item $w_{ip}^s w_{jp}^s= w_{ij}^s(w_{jp}^s-w_{ip}^s) \quad \text{for } 1\leq s \leq r_n \text{ and } 1\leq i < j <p \leq m+n.$    
\item $w_{ij}^s = w_{ij}^{s'} \quad \text{for } 1\leq s, s' \leq r_n \text{ and } 1\leq i < j \leq m$,
\item $w^s_{ij} = w^{s'}_{ij} \quad \text{for }  r_{n_u}\leq s, s'\leq r_n \text{ and } 1\leq i \leq m + n_{u}-1, ~ m+n_{u-1} +1\leq j \leq m + n_u,$ where $u= 1, \ldots, \ell-1$.
\end{enumerate}
\end{proposition}

\begin{proof}
For $1\leq s \leq r_n$, consider the projection map $q_s : E_B^{\br} \to E$ defined as 
$$(e_1, \ldots, e_{r_n})\mapsto e_s.$$
One can realize an element $e_s = (o_1, \ldots, o_m, z^s_1, \ldots, z^s_n)$ in $E$ as $(x_1, \ldots, x_{m+n})$ where $x_i = o_i$ for $1\leq i\leq m$ and $x_i = z^s_{i-m}$ for $m+1\leq i \leq m+n$. Therefore, the cohomology class $w_{ij}\in H^{d-1}(E)$ induces a cohomology class $w^s_{ij}$ in $H^{d-1}(E_B^{\br})$ defined as $$w^s_{ij} := (q_{s})^*(w_{ij}).$$
So, the result follows from Lemma \ref{Lem_cohomology ring of E} and the relations in \eqref{Eq_description of total space}.
\end{proof}

For $p\geq 0$, we consider $I = (i_1, i_2, \dots , i_p)$ and $J = (j_1, j_2, \dots , j_p)$ where $i_b, j_b \in \{1,\dots, m +n\}$ for $b = 1,\dots, p$. We say $I<J$ if and only if $i_b< j_b$ for all $b=1,\dots, p$. Moreover, $J$ is called increasing if and only if $j_1 < \cdots < j_p$.  
We denote the cohomology class $w^s_{i_1j_1} w^s_{i_2j_2} \cdots w^s_{i_pj_p}\in H^{p(d-1)}(E^{\br}_B)$ by $w^s_{IJ}$, where $I$ and $J$  are any $p$-tuple with $I<J$ and $1\leq s \leq r_n$, i.e, $$w^s_{IJ} := w^s_{i_1j_1} w^s_{i_2j_2} \cdots w^s_{i_pj_p}.$$
If $p=0$, then we set $w^s_{IJ} = 1$.
Further, for $u =0, 1, \ldots, \ell-1$, we consider 
\begin{equation}\label{Eq_w^a}
\bar{w}^u := \prod^{r_{n_{u+1}}}_{s = r_{n_{u}}+1} w^s_{I_sJ_s}; \text{ where } I_s < J_s ~~(\text { assume } r_{n_0} = 0).
\end{equation}
It follows that $\bar{w}^u\in H^{D}(E^{\br}_B)$, where $D = \displaystyle\sum^{r_{n_u+1}}_{s=r_{n_u}+1} |I_s|(d-1)$, where $|I_s|$ denotes the length of the tuple $I_s$. Clearly, for each fixed $u\in \{0,\ldots,\ell-1\}$, the cohomology classes $\bar{w}^u$ depends on the lengths of $I_s$. Moreover, for all $1\leq s\leq r_n$, the cohomology classes $w^s_{IJ}$ are all equal to $w_{IJ}$ whenever $J$ takes values in $\{2, \ldots, m\}$, and we denote this class by $\bar{w}$ in $H^{|J|}(E^{\br}_B)$.

\begin{proposition}\label{Prop_additive basis}
Let the tuples $J, J_1, J_2, \ldots, J_{r_n}$ be increasing such that $J_s$ take values in $\{m + n_u + 1, \ldots, m + n\}$ whenever $r_{n_u}+1\leq s\leq r_{n_{u+1}}$ for $u = 0,1,\ldots, \ell-1$. Then an additive basis of $H^*(E^{\br}_B)$ is formed by the following type of cohomology classes: $$\bar{w} \bar{w}^0 \bar{w}^1 \cdots \bar{w}^{\ell-1},$$ where $\bar{w}^u$ is defined in \eqref{Eq_w^a} for $u=0, 1, \dots, \ell-1$ and $\bar{w}=w_{IJ}$ when $J$ takes values in $\{2, \dots, m\}$.
\end{proposition}

\begin{proof}
We use Leray-Hirsch theorem to the fibration $\hat{p}\colon E^{\br}_B\to B$ defined in \eqref{Eq_hat p}.  The cohomology classes $w^s_{ij}$ with $1\leq i< j \leq m$ originate from the base $B$ of this fibration $\hat{p}$. Moreover, using the cohomology algebra of $B$ (see Lemma \ref{Lem_cohomology ring of E}), we say that an additive basis of $H^*(B)$ consists of the cohomology classes $w_{IJ}$, where $J$ is increasing and takes values in $\{2,...,m\}$. 

For $u=0,\ldots,\ell-1$, recall the fibre $$X_u\cong F(\rr^d-\{o_1,\ldots,o_m,z_1,\ldots,z_{n_u}\}, n-n_u)$$ of the fibration $p_u\colon E\to F(\rr^d,m+n_u)$ defined in \eqref{Eq_FNB}  and \eqref{Eq_pu}.
Remark that $p_0 = p$ and $X_0 = X$.
Using the known results about the cohomology algebra of configuration spaces \cite[Chapter V, Theorem 4.2 and 4.3]{FadH01}, each $H^*(X_u)$ is free and additively generated by the restriction of the cohomology classes $w_{I'J'}$ over the fibre $X_u$ of $p_u$, where $J'$ is increasing and takes values in $\{m+n_u+1,\ldots, m+n\}$. Therefore, for the fibre $\hat{p}^{-1}\{(o_1,\dots, o_m)\}$, we have $$H^*(\hat{p}^{-1}\{(o_1, \dots, o_m)\}) \cong H^*(X^{r_{n_1}}) \otimes H^*(X^{(r_{n_2} - r_{n_1})}_1) \otimes \cdots \otimes H^*(X^{(r_{n_\ell} - r_{n_{\ell - 1}})}_{\ell - 1}).$$
Further, applying the K\"unneth theorem, it follows that the restriction of the family of the cohomology classes $\bar{w}^u=\displaystyle\prod^{r_{n_{u+1}}}_{s = r_{n_{u}}+1} w^s_{I_sJ_s}$ onto the space $X^{(r_{n_{u+1}}-r_{n_u})}_u$ form a free basis of $H^*(X^{(r_{n_{u+1}}-r_{n_u})}_u)$, provided each $J_s$ is increasing and takes values in $\{m+n_u+1,\ldots,m+n\}$. Therefore, the restrictions of the family of classes  $\bar{w}^0 \bar{w}^1 \cdots \bar{w}^{\ell-1}$ onto the fibre $\hat{p}^{-1}\{(o_1,\ldots,o_m)\}$ form an additive basis of the cohomology $H^*(\hat{p}^{-1}\{(o_1, \dots, o_2)\})$. Hence, applying the Leray-Hirsch theorem, we obtain an additive basis of $H^*(E^{\br}_B)$ given by the cohomology classes of the form $\bar{w} \bar{w}^0 \bar{w}^1 \cdots \bar{w}^{\ell-1}.$
\end{proof}

Recall the diagonal map $\Delta\colon E\to E^{\br}_B$ defined above in Proposition \ref{Prop_lbptc}. The following proposition follows from the definition of cohomology classes $w^s_{ij}\in H^*(E^{\br}_B)$.

\begin{proposition}\label{Prop_kernel of diagonal map}
For the homomorphism $\Delta^*\colon H^*(E^{\br}_B)\to H^*(E)$ induced by the diagonal map $\Delta\colon E\to E^{\br}_B$, the kernel of $\Delta^*$ contains the following cohomology classes $$w^s_{ij} - w^{s'}_{ij},$$ where $1\leq s, s'\leq r_n$ and $1\leq i<j\leq m+n$.
\end{proposition}
\noindent We keep the same notation for the cohomology classes of $H^{*}(E^{\br}_B)$ in the subsequent sections.

\sect{Topological complexity in odd-dimensional Euclidean space}\label{sec:odd}
In this section, we investigate the main problem for the odd-dimensional case, as stated in Section \ref{sec: variable target}.  The following theorem is one of the main results of our paper.
\begin{theorem}\label{Th_odd case}
The $\bar{\textbf{r}}$-sequential parametrized topological complexity of the Fadell-Neuwirth fibration \eqref{Eq_FNB} is given by $$\TC_{\br}[p\colon F(\mathbb{R}^d, n+ m) \to F(\mathbb{R}^d, m)] =\sum^{n}_{i=1} {r_i} + m -1,$$
where $d\geq 3$ is odd, $n\geq 1$, and $m\geq 2$.
\end{theorem}
We obtained an upper bound for $\TC_{\br}[p\colon F(\mathbb{R}^d, n+ m) \to F(\mathbb{R}^d, m)]$ in Proposition \ref{Prop_ubptc}. To complete the theorem, we now need to establish a lower bound, which we provide in Proposition \ref{Prop_odd lower bound} that matches the given upper bound.

\begin{remark}
In Section \ref{sec:problem}, we obtained the result $\TC_{\br}[p\colon F(\mathbb{R}^d, m+n)\to F(\mathbb{R}^d,m)]=6$, where two robots move through two and three points, respectively, in the presence of two obstacles in $\mathbb{R}^3$ (that is, $d=3, ~m = n =2, ~\br =(r_1, r_2)= (2,3))$. 
\end{remark}

\begin{proposition}\label{Prop_odd lower bound}
For an odd integer $d\geq 3$ and $m\geq 2, n\geq 1$, we have the following relation:  $$\TC_{\br}[p \colon F(\mathbb{R}^d, m+n)\to F(\mathbb{R}^d, m)] \geq \displaystyle\sum^{n}_{i=1} {r_i} + m -1.$$
\end{proposition}

\begin{proof}
Following Proposition \ref{Prop_kernel of diagonal map}, we take some cohomology classes having the form $w^s_{ij} - w^{s'}_{ij}$ and use Proposition \ref{Prop_lbptc} to conclude our result.
Consider the classes
\begin{align*}
x = & \prod^m_{i=2}(w^1_{i(m+1)} - w^2_{i(m+1)}),\\
x^0 = & \prod^{m+n_1}_{j=m+1}(w^2_{1j} - w^1_{1j}) \prod^{r_{n_1}}_{s=2} ~\prod^{m+n}_{j=m+1} (w^s_{1j} - w^1_{1j}),\\
x^1 = & \prod^{m+n_2}_{j=m+n_1+1}(w^{r_{n_1}+1}_{1j} - w^1_{1j}) \prod^{r_{n_2}}_{s=r_{n_1}+1} ~\prod^{m+n}_{j=m+ n_1 + 1} (w^s_{1j} - w^1_{1j}),\\
\vdots \\
x^{\ell-1} = & \prod^{m+n}_{j = m+n_{\ell-1}+1}(w^{r_{n_{\ell-1}}+1}_{1j} - w^1_{1j}) \prod^{r_{n_\ell}}_{s=r_{n_{\ell-1}}+1} ~\prod^{m+n}_{j=m+ n_{\ell-1} + 1} (w^s_{1j} - w^1_{1j}).
\end{align*}
Since the degree of any cohomology class is even, the product of two cohomology classes is commutative, i.e., $w^s_{ij} w^{s'}_{i'j'} = w^{s'}_{i'j'} w^s_{ij}$ for all $i,j, i',j',s,s'$. Observe that  
$$x = \sum_{I,I'} (-1)^{|I'|} w^1_{IJ} w^2_{I'J'},$$ where the summation is taken over all the finite tuples $I = (i_1, \ldots , i_p), I' =(i'_1,\ldots, i'_q)$ such that $\{i_1, \ldots, i_p\} \cap \{i'_1, \ldots, i'_q\} = \emptyset$ and  $\{i_1, \ldots, i_p\} \cup \{i'_1, \ldots, i'_q\} = \{2,...,m\}$. Also, $J=(m+1, \ldots, m+1)$ and $J'=(m+1, \ldots, m+1)$ satisfying $0\leq |J|=|I|\leq m-1$ and $0\leq |J'|=|I'| \leq m-1 $.

After a thorough calculation, one obtains the following result. We omit the detailed computation here to maintain the flow for the reader. A brief outline is provided in Appendix \ref{Appendix for odd case}. We have the product of $\sum^{n}_{i=1} {r_i} + m -1$ many cohomology classes as
\begin{align*}
x x^0\cdots x^{\ell-1} = & (-2)^n \sum_{I,I'} (-1)^{|I'|} w^1_{IJ} w^2_{I'J'} \prod^{m+n}_{j=m+1} w^1_{1j}\cdots w^{r_{n_1}}_{1j}\\
&\hspace{2cm} \prod^{m+n}_{j=m+n_1+1} w^{r_{n_1}+1}_{1j} \cdots w^{r_{n_2}}_{1j} \cdots \prod^{m+n}_{j=m+n_{\ell-1}+1} w^{r_{\ell-1}+1}_{1j}\cdots w^{r_n}_{1j}\\
= & (-2)^n \sum_{I,I'} (-1)^{|I'|} w^1_{IJ} w^2_{I'J'} w^1_{1(m+1)} w^2_{1(m+1)} \prod^{m+n}_{j=m+2} w^1_{1j} w^2_{1j} \prod^{m+n}_{j=m+1} w^3_{1j}\cdots w^{r_{n_1}}_{1j}\\
&\hspace{2cm} \prod^{m+n}_{j=m+n_1+1} w^{r_{n_1}+1}_{1j} \cdots w^{r_{n_2}}_{1j} \cdots \prod^{m+n}_{j=m+n_{\ell-1}+1} w^{r_{\ell-1}+1}_{1j}\cdots w^{r_n}_{1j}.
\end{align*}
Using Proposition \ref{Prop_relations on cohomology classes}, we can expand the expression $\displaystyle\sum_{I,I'} (-1)^{|I'|} w^1_{IJ} w^2_{I'J'} w^1_{1(m+1)} w^2_{1(m+1)}$ to basis elements of the form as in Proposition \ref{Prop_additive basis} and show that the cohomology class $w_{12} w_{23}\cdots w_{2m} w^1_{2(m+1)} w^2_{1(m+1)}$ occurs only once in the obtained expression.
Observe that if $I = (2,\ldots,m)$, then the term $$w_{12} w_{23}\cdots w_{2m} w^1_{2(m+1)} w^2_{1(m+1)}$$ appears exactly once in $w^1_{IJ} w^1_{1(m+1)} w^2_{1(m+1)}$. Also, the term $w_{12} w_{23}\cdots w_{2m} w^1_{2(m+1)} w^2_{1(m+1)}$ does not appear in $w^2_{I'J'} w^1_{1(m+1)} w^2_{1(m+1)}$ whenever $I' = (2,\ldots,m)$. 
Moreover, if $2\in I$ and $j\in I'$ then the term $w_{2j}$ does not appear as part of any element in $w^1_{IJ} w^2_{I'J'} w^1_{1(m+1)} w^2_{1(m+1)}$. A similar situation arises whenever $2\in I'$ and $j\in I$. See Appendix \ref{Appendix odd only once} for the detailed calculation.

It follows that the product $x x^0 x^1 \cdots x^{\ell-1}$ contains the following basis element with nonzero coefficient: 
\begin{align*}
& w_{12} w_{23}\cdots w_{2m} w^1_{2(m+1)} w^2_{1(m+1)}\prod^{m+n}_{j=m+2} w^1_{1j} w^2_{1j} \prod^{m+n}_{j=m+1} w^3_{1j}\cdots w^{r_{n_1}}_{1j} \\&\hspace{5cm}\prod^{m+n}_{j=m+n_1+1} w^{r_{n_1}+1}_{1j} \cdots w^{r_{n_2}}_{1j} \cdots \prod^{m+n}_{j=m+n_{\ell-1}+1} w^{r_{\ell-1}+1}_{1j}\cdots w^{r_n}_{1j}.\\
& = w_{IJ}w^1_{I_1J_1}w^2_{I_2J_2}\prod^{m+n}_{j=m+1} w^3_{1j}\cdots w^{r_{n_1}}_{1j} \prod^{m+n}_{j=m+n_1+1} w^{r_{n_1}+1}_{1j} \cdots w^{r_{n_2}}_{1j} \cdots \prod^{m+n}_{j=m+n_{\ell-1}+1} w^{r_{\ell-1}+1}_{1j}\cdots w^{r_n}_{1j}\\
& = w_{IJ} \prod^{\ell-1}_{u=0} \prod^{r_{n_{u+1}}}_{s = r_{n_{u}}+1} w^s_{I_sJ_s},
\end{align*}
where 
\begin{enumerate}[(i)]
\item $I= (1,2,\ldots,2), ~J=(2,3,\ldots,m)$,
\item  $I_1=(2,1, \ldots,1), ~I_s = (1,\ldots,1)$ for $2\leq s\leq r_n$. 
\item $J_s = (m + n_u + 1, \ldots, m + n)$ whenever $r_{n_u}+1\leq s\leq r_{n_{u+1}}$ for $u = 0,1,\ldots, \ell-1$. 
\end{enumerate}
Clearly, this is a basis element, so $x x^0 \cdots x^{\ell-1}\neq 0$. Hence, we get the desired result.
\end{proof}

\sect{Topological complexity in even-dimensional Euclidean space}\label{sec:even}
 
In the following proposition, we show that for the Fadell–Neuwirth fibration in \eqref{Eq_FNB}, the lower bound of $\TC_{\br}[p\colon F(\mathbb{R}^d, m+n)\to F(\mathbb{R}^d, m)]$ is reduced by one relative to the bound established in Proposition \ref{Prop_odd lower bound}, whenever the dimension $d$ of the ambient Euclidean space is even.

\begin{proposition}\label{proposition lower bound even case}
For an even integer $d\geq 2$ and $m\geq 2, n\geq 1$, we have the following relation:
$$\TC_{\br}[p \colon F(\mathbb{R}^d, m+n)\to F(\mathbb{R}^d, m)] \geq \displaystyle\sum^{n}_{i=1} {r_i} + m -2.$$
\end{proposition}

\begin{proof}
To apply Proposition \ref{Prop_lbptc}, we consider the following cohomology classes 
\begin{align*}
y = & \prod^m_{i=2}(w^1_{i(m+1)} - w^2_{i(m+1)}),\\
y' = & \prod^{m+n}_{j=m+2} (w^2_{(j-1)j} - w^1_{(j-1)j}),\\
y^0 = & \prod^{r_{n_1}}_{s=2} ~\prod^{m+n}_{j=m+1} (w^s_{1j} - w^1_{1j}),\\
y^1 = & \prod^{r_{n_2}}_{s=r_{n_1}+1} ~\prod^{m+n}_{j=m+ n_1 + 1} (w^s_{1j} - w^1_{1j}),\\
\vdots \\
y^{\ell-1} = & \prod^{r_{n_\ell}}_{s=r_{n_{\ell-1}}+1} ~\prod^{m+n}_{j=m+ n_{\ell-1} + 1} (w^s_{1j} - w^1_{1j}).
\end{align*}
Since the cohomology classes are in odd degree, 
\begin{equation}\label{Eq_anticommutativity}
w^s_{ij} w^{s'}_{i'j'} = - w^{s'}_{i'j'} w^s_{ij} ~~\text{ for all } i,j, i',j',s,s'.
\end{equation}
 In the rest of this section, we use the $\pm$ sign to denote either $+$ or $-$ without any ambiguity.
For $u=1,\ldots, \ell-1$, we can simplify $y^u$ by expanding the product on $s$ from $r_{n_u}+1$ to $r_{n_u+1}$ and yield
\begin{equation*}
  y^u = \prod^{m+n}_{j=m+n_u +1} \Big(w^{r_{n_u} +1}_{1j} \cdots w^{r_{n_{u+1}}}_{1j} \pm w^1_{1j} \sum^{r_{n_{u+1}}}_{a=r_{n_u}+1} w^{r_{n_u} +1}_{1j}\cdots \hat{w}^a_{1j} \cdots w^{r_{n_{u+1}}}_{1j}\Big),  
\end{equation*} 
where $\hat{w}^a_{1j}$ denotes that the term ${w}^a_{1j}$ is absent in the product and $\pm$ indicates that the sign of an element can be either positive or negative inside the summation.
A similar simplification provides $$y^0 = \prod^{m+n}_{j=m+1} (w^2_{1j} - w^1_{1j}) \prod^{m+n}_{j=m+1}\Big(w^3_{1j} \cdots w^{r_{n_1}}_{1j} \pm w^1_{1j} \sum^{r_{n_1}}_{a=3} w^3_{1j} \cdots \hat{w}^a_{1j} \cdots w^{r_{n_1}}_{1j}\Big).$$
To simplify the product of the cohomology classes, we rearrange some of the cohomology classes. First, we remove $\displaystyle\prod^{m+n}_{j=m+1} (w^2_{1j} - w^1_{1j})$ from $y^0$ to obtain $\tilde{y}^0$. Further, we divide the product $\displaystyle\prod^{m+n}_{j=m+1} (w^2_{1j} - w^1_{1j})$ into two terms $(w^2_{1(m+1)} - w^1_{1(m+1)})$ and $\displaystyle\prod^{m+n}_{j=m+2} (w^2_{1j} - w^1_{1j})$. 
Then we multiply $y$ with the first term $(w^2_{1(m+1)} - w^1_{1(m+1)})$ to obtain $\tilde{y}$ and multiply the remaining term $\displaystyle\prod^{m+n}_{j=m+2} (w^2_{1j} - w^1_{1j})$ with $y'$ to produce $\tilde{y}'$. 
Therefore, we obtain
\begin{align}\label{elb}
\tilde{y}^0 = & \prod^{m+n}_{j=m+1}\Big(w^3_{1j} \cdots w^{r_{n_1}}_{1j} \pm w^1_{1j} \sum^{r_{n_1}}_{a=3} w^3_{1j} \cdots \hat{w}^a_{1j} \cdots w^{r_{n_1}}_{1j}\Big),\nonumber \\
 \tilde{y}' = & \prod^{m+n}_{j=m+2} (w^2_{1j} - w^1_{1j}) \prod^{m+n}_{j=m+2} (w^2_{(j-1)j} - w^1_{(j-1)j}), \nonumber \\
= & \prod^{m+n}_{j=m+2} \Big(w^2_{1(j-1)} w^2_{(j-1)j} - w^2_{1(j-1)} w^2_{1j} - w^1_{1j} w^2_{(j-1)j} + w^1_{(j-1)j} w^2_{1j} + \nonumber \\
& \hspace{3cm} w^1_{1(j-1)} w^1_{(j-1)j} - w^1_{1(j-1)} w^1_{1j}\Big) (\text{using \eqref{Eq_anticommutativity} and Proposition } \ref{Prop_relations on cohomology classes}(b)),\nonumber \\
\tilde{y} =&  ~(w^2_{1(m+1)} - w^1_{1(m+1)}) \prod^m_{i=2}(w^1_{i(m+1)} - w^2_{i(m+1)}) =  \pm \sum_{I,I'} w^1_{IJ} w^2_{I'J'},
\end{align}
where the summation in \eqref{elb} is taken over all the finite tuples $I = (i_1, \ldots , i_p), I' =(i'_1,\ldots, i'_q)$ such that $\{i_1, \ldots, i_p\} \cap \{i'_1, \ldots, i'_q\} = \emptyset$ and  $\{i_1, \ldots, i_p\} \cup \{i'_1, \ldots, i'_q\} = \{1, ..., m\}$. Also, $J=(m+1, \ldots, m+1)$ and $J'=(m+1, \ldots, m+1)$ satisfying $0\leq |J|=|I|\leq m$ and $0\leq |J'|=|I'| \leq m $.
Now we consider the product
\begin{align*}
\tilde{y}^0 y^1 \cdots y^{\ell-1} = & \prod^{m+n_1}_{j=m+1}\Big(w^3_{1j} \cdots w^{r_{n_1}}_{1j} \pm w^1_{1j} \sum^{r_{n_1}}_{a=3} w^3_{1j} \cdots \hat{w}^a_{1j} \cdots w^{r_{n_1}}_{1j}\Big)\\
& \prod^{m+n_2}_{j=m+n_1 +1} \Big(w^3_{1j} \cdots w^{r_{n_2}}_{1j} \pm w^1_{1j} \sum^{r_{n_2}}_{a=3} w^3_{1j} \cdots \hat{w}^a_{1j} \cdots w^{r_{n_2}}_{1j}\Big) \cdots \\
& \prod^{m+n}_{j=m+n_{\ell-1} +1}\Big(w^3_{1j} \cdots w^{r_{n}}_{1j} \pm w^1_{1j} \sum^{r_{n}}_{a=3} w^3_{1j} \cdots \hat{w}^a_{1j} \cdots w^{r_{n}}_{1j}\Big).
\end{align*}
With a rigorous calculation performed in Appendix \ref{app_proposition lower bound even case}, we conclude that the product $\tilde{y}' (\tilde{y}^0 y^1 \cdots y^{\ell-1})$ contains the following product term 
\begin{equation}\label{eut}
\begin{split}
&\Big(w^3_{1(m+1)} \cdots w^{r_{n_1}}_{1(m+1)} \pm w^1_{1(m+1)} \sum^{r_{n_1}}_{a=3} w^3_{1(m+1)} \cdots \hat{w}^a_{1(m+1)} \cdots w^{r_{n_1}}_{1(m+1)}\Big)\\
&\hspace{.5cm} \Big{(} \prod^{m+n_1}_{j=m+2} w^1_{1j} w^2_{(j-1)j} w^3_{1j} \cdots w^{r_{n_1}}_{1j} \prod^{m+n_2}_{j=m+ n_1 +1} w^1_{1j} w^2_{(j-1)j} w^3_{1j} \cdots w^{r_{n_2}}_{1j} \cdots \\
& \hspace{7cm} \prod^{m+n}_{j=m+n_{\ell-1} +1} w^1_{1j} w^2_{(j-1)j} w^3_{1j} \cdots w^{r_{n}}_{1j}\Big{)}
\end{split}
\end{equation}
only once. We skip the calculation here for the sake of fluency in argument. 
For $I, I' , J, J'$ in \eqref{elb}, multiplying the product $w^1_{IJ} w^2_{I'J'}$ with the elements in $\tilde{y}' \tilde{y}^0 y^1 \cdots y^{\ell-1}$, except the element in \eqref{eut}, does not yield any element of the form $w^1_{IJ} w^2_{I'J'} \Gamma$, where $\Gamma$ is the element in \eqref{eut}.
Therefore, $\tilde{y}\tilde{y}' \tilde{y}^0 y^1 \cdots y^{\ell-1}$ contains the following products precisely once:
\begin{equation*} 
\begin{split}
&\sum_{I,I'} w^1_{IJ} w^2_{I'J'} \Big(w^3_{1(m+1)} \cdots w^{r_{n_1}}_{1(m+1)} \pm w^1_{1(m+1)} \sum^{r_{n_1}}_{a=3} w^3_{1(m+1)} \cdots \hat{w}^a_{1(m+1)} \cdots w^{r_{n_1}}_{1(m+1)}\Big)
\\
& \hspace{1cm} \Big( \prod^{m+n_1}_{j=m+2} w^1_{1j} w^2_{(j-1)j} w^3_{1j} \cdots w^{r_{n_1}}_{1j} \prod^{m+n_2}_{j=m+ n_1 +1} w^1_{1j} w^2_{(j-1)j} w^3_{1j} \cdots w^{r_{n_2}}_{1j} \cdots
\\ & \hspace{3cm} \prod^{m+n}_{j=m+n_{\ell-1} +1} w^1_{1j} w^2_{(j-1)j} w^3_{1j} \cdots w^{r_{n}}_{1j}. \Big)
\end{split}
\end{equation*}
For $I = (1,\ldots, m)$ the product $\tilde{y}\tilde{y}' \tilde{y}^0 y^1 \cdots y^{\ell-1}$ contains the following term:
\begin{align*}
& w_{12} w_{23} \cdots w_{2m} w^1_{m(m+1)} w^3_{1(m+1)} \cdots w^{r_{n_1}}_{1(m+1)} \prod^{m+n_1}_{j=m+2} w^1_{1j} w^2_{(j-1)j} w^3_{1j} \cdots w^{r_{n_1}}_{1j}\\
&\hspace{1cm} \prod^{m+n_2}_{j=m+ n_1 +1} w^1_{1j} w^2_{(j-1)j} w^3_{1j} \cdots w^{r_{n_2}}_{1j} \cdots \prod^{m+n}_{j=m+n_{\ell-1} +1} w^1_{1j} w^2_{(j-1)j} w^3_{1j} \cdots w^{r_{n}}_{1j}\\
= &\pm w_{12} w_{23} \cdots w_{2m} w^1_{m(m+1)} w^1_{1(m+2)}\cdots w^1_{1(m+n)} \prod^{m+n}_{j=m+2} w^2_{(j-1)j}  \prod^{m+n}_{j=m+1} w^3_{1j}\cdots w^{r_{n_1}}_{1j}\\
&\hspace{2cm} \prod^{m+n}_{j=m+ n_1 +1} w^{r_{n_1}+1}_{1j} \cdots w^{r_{n_2}}_{1j} \cdots \prod^{m+n}_{j=m+n_{\ell-1}+1} w^{r_{\ell-1}+1}_{1j}\cdots w^{r_n}_{1j}\\
= & \pm w_{IJ} \prod^{\ell-1}_{u=0} \prod^{r_{n_{u+1}}}_{s = r_{n_{u}}+1} w^s_{I_sJ_s};
\end{align*}
where 
\begin{enumerate}[(i)]
\item $I= (1,2,\ldots,2), ~J=(2,3,\ldots,m)$,
\item  $I_1=(m,1, \ldots,1), I_2 = (m+1, \ldots, m+n-1), ~I_s = (1,\ldots,1)$ for $3\leq s\leq r_n$. 
\item $J_2 =(m+2, \ldots m+n),~J_s = (m + n_u + 1, \ldots, m + n)$ whenever $r_{n_u}+1\leq s(\neq 2)\leq r_{n_{u+1}}$ for $u = 0,1,\ldots, \ell-1$. 
\end{enumerate}
Using the similar argument as in the proof of Proposition \ref{Prop_odd lower bound}, this term is the unique nonzero basis element in $\tilde{y}\tilde{y}' \tilde{y}^0 y^1 \cdots y^{\ell-1}$. Thus $\tilde{y}\tilde{y}' \tilde{y}^0 y^1 \cdots y^{\ell-1} \neq 0$ which implies $y y' y^0 y^1\cdots y^{\ell-1} \neq 0$. Since this product contains $\sum^{n}_{i=1} {r_i} + m -2$ many cohomology classes, the result follows.
\end{proof}

\section{An algorithm}\label{sec:algorithm}
In this section, we present an algorithm for collision-free motion of $n$ robots, $z_1, z_2, \dots, z_n$, where robot $z_i$ has to sequentially visit $r_i$ many prescribed states in presence of $m\geq 2$ obstacles with unknown a priori positions.
We first present a general algorithm that works for both odd and even-dimensional cases consisting $R + m$ local algorithms, where $R=\sum^{n}_{i=1} {r_i}$. Then for the even-dimensional case we reduce the number of local algorithms to $R + m -1$. It proves that for even $d$, the upper bound of the $\bar{\textbf{r}}$-th sequential parametrized topological complexity of the Fadell-Neuwirth fibration $p\colon E\to B$ as in \eqref{Eq_FNB} is given by $$\TC_{\br}[p\colon E \to B] \leq R + m -2.$$
The idea of the algorithm is similar to that in \cite{FarP23}, where the authors provide an algorithm only for the even-dimensional case, while we present algorithms for cases of both dimensions. 
Here we note that the algorithm \cite[Section 2]{FarP23} works for the odd-dimensional case as well, i.e., if a single robot moves in $\rr^d$ for any $d\geq 2$, avoiding collisions with $m\geq 2$ obstacles, then we may apply that method. Moreover, for the odd-dimensional case we do not need to vary the line $L_b$ depending on the position of the obstacles. 
We apply \cite[Section 2]{FarP23} later in this section.

We fix an oriented line $L$ passing through the origin in $\rr^d$. Let $e$ be the unit vector in the positive direction of the line $L$, and we fix the unit vector $e^{\perp}$ perpendicular to $e$. We denote by $q: \rr^d \to L$ the orthogonal projection on the line $L$, defined by $q(x)=\langle  x, e \rangle  \cdot  e$.

Consider the space $E^{\bar{\textbf{r}}}_B$. For any configuration $C\in E^{\bar{\textbf{r}}}_B$
\begin{equation*}
C=(o_1, \ldots, o_m, z^1_{1}, \ldots , z^{r_{1}}_{1}, \ldots, z^1_{n}, \ldots , z^{r_{n}}_{n} ),
\end{equation*} 
we shall denote $q(C)$ the set of the projection points $$q(C) = \{q(o_j), ~q(z_i^{k_i}); ~1\leq j \leq m,~ 1\leq i \leq n, ~ 1\leq k_i \leq r_i\}$$ in their respective order.
Since some of the projection points may coincide, the cardinality of this set $ |q(C)|$ may vary between $1$ and $R + m$ depending on the configuration $C$, i.e. $$1\leq |q(C)|\leq R + m.$$

\subsection{Partition of \texorpdfstring{$E^{\br}_B$}:}

For any $c \in \{1,\dots,R + m\}$, we define the set $$W_c=\big \{ C \in E^{\bar{\textbf{r}}}_B \colon |q(C)| = c\}.$$
We aim to present an algorithm over each $W_c$. On the other hand, we define $$E_{\mu,\nu}= \big\{C\in E^{\bar{\textbf{r}}}_B : |q(C)| = \mu+\nu ~\text{and}~ |q(O)| =\mu \big\},$$ where $O$ is the configuration of the obstacles $(o_1, \dots, o_m)$ associated to the configuration $C$ and $1\leq \mu \leq m,~ 0\leq \nu \leq R$. Notice that $$W_c=\bigsqcup_{\mu+\nu=c}E_{\mu,\nu}.$$
It is easy to check that for any $\mu$ and $\nu$, $$\overline{E_{\mu,\nu}}\subset \bigcup_{\mu'\leq \mu, \nu'\leq \nu}E_{\mu', \nu'}.$$ So, if $\mu+\nu=c$ then $\overline{E_{\mu, \nu}}\cap W_c=E_{\mu, \nu}$. Therefore, each $E_{\mu, \nu}$ satisfying $\mu+\nu=c$ is closed as well as open in $W_c$. Hence, a continuous algorithm on each $E_{\mu,\nu}$ collectively defines a continuous algorithm on $W_c$. In the following, we construct a continuous algorithm over each $E_{\mu,\nu}$, where $ 1\leq \mu \leq m,~ 0\leq \nu \leq R$. 

\subsection{Decomposition of the set \texorpdfstring{$E_{\mu, \nu}$}:}
Let $P$ be a finite collection of $R+m$ symbols as follows
\begin{equation}\label{eq symbols}
P=\big\{o_1, \ldots, o_m, z^1_{1}, \ldots , z^{r_{1}}_{1}, \ldots, z^1_{n}, \ldots , z^{r_{n}}_{n} \big\}.
\end{equation} A binary relation $\leq$ on $P$ is called \emph{quasi-order} if it is reflexive and transitive. A quasi-order $\leq$ is called \emph{linear} if for any $a, b \in P$ either $a\leq b$ or $b\leq a$. Note that a quasi-order or a linear quasi-order allows for distinct elements $a, b \in P$ to satisfy $a\leq b$ and $b\leq a$. 
For any linear quasi-order on $P$ we can define an equivalence relation, $a\sim b$ if and only if $a\leq b$ and $b\leq a$. If $a\sim b$ then we say that $a$ and $b$ are equivalent with respect to the quasi-order $\leq$. We denote by $\Sigma_{\mu, \nu}$, the set of all linear quasi-orders over $P$ having in total $\mu+\nu$ many equivalence classes such that the sub-collection $\{o_1, \ldots, o_m\}$ has $\mu$ equivalence classes.

Let $C\in E^{\bar{\textbf{r}}}_B$ be a  configuration. For any two elements $a$ and $b$ of $C$, we say $a\leq b$ if and only if $q(a)\leq q(b)$. Then $\leq$ associates a linear quasi-order over $P$. In this case, we say that $C$ generates the quasi-order $\leq$ over $P$. For any $\sigma \in \Sigma_{\mu, \nu}$, we define $$E_{\mu, \nu}^{\sigma} \coloneq \big\{ C\in E_{\mu, \nu} ~\colon~ C \text{ generates the quasi-order } \sigma\big\}.$$

\noindent Clearly, $$E_{\mu, \nu}=\bigsqcup_{\sigma \in \Sigma_{\mu, \nu}} E_{\mu, \nu}^{\sigma}.$$
Also note that each $E_{\mu, \nu}^{\sigma}$ is closed and open in $E_{\mu, \nu}$. Therefore, a continuous algorithm on each $E_{\mu,\nu}^{\sigma}$ collectively defines a continuous algorithm on $E_{\mu, \nu}$. 

In the following, we define an algorithm over each $E_{\mu,R}^{\tau}$, where $\tau \in \Sigma_{\mu,R}$ and $1\leq \mu \leq m.$ Then we induce an algorithm over each $E_{\mu,\nu}^{\sigma}$ from the algorithm over $E_{\mu,R}^{\tau}$ for some $\tau \in \Sigma_{\mu,R}$.

\subsection{Algorithm over \texorpdfstring{$E_{\mu, R}^{\tau}$}:}\label{Algorithm over E u,R,tau} 
 An algorithm over $E_{\mu, R}^{\tau}$ is a section of the fibration $\Pi_{\bar{\textbf{r}}}$ in \eqref{Eq_pi r bar}. It is determined by collection of $n$ paths $$\gamma_1^C, \gamma_2^C, \ldots, \gamma_n^C : I \to \rr^d,$$
continuously depending on the configuration $C\in E_{\mu, R}^{\tau}$ and satisfying the following properties:
\begin{enumerate}[(i)]
    \item $\gamma_i^C(t)\neq \gamma_j^C(t)$ for $i\neq j$,
    \item $\gamma_i^C(t)\neq o_j$ for $i=1, \ldots, n$ and $j=1, \ldots, m$,
    \item $\big(\gamma_i^C(t_1), \gamma_i^C(t_2), \ldots, \gamma_i^C(t_{r_i})\big)= \big(z_i^1, z_i^2, \ldots, z_i^{r_i}\big)$  for $i=1, \ldots, n$.   
\end{enumerate}

\noindent  The time schedule for our problem is given by
\[
0 = t_1 < t_2 < \cdots < t_{r_n} = 1.
\]
We construct the paths $\gamma_1^C, \ldots, \gamma_n^C$ for robots 
$z_1, \ldots, z_n$ respectively, as follows.

\medskip
\noindent
\subsection*{The paths over the first interval \texorpdfstring{$[t_1, t_2]$}:}

We divide $[t_1, t_2]$ into $n$ equal subintervals:
$$t_1 = t_{1,0} < t_{1,1} < \cdots < t_{1,n} = t_2.$$
\begin{itemize}
    \item On $[t_{1,0}, t_{1,1}]$, we apply the method of \cite[Section~2]{FarP23} to move the robot $z_1$ from $z_1^1$ to $z_1^2$, treating $o_1, \ldots, o_m$ and $z_2^1, \ldots, z_n^1$ as obstacles. During this subinterval, the paths $\gamma_2^C, \ldots, \gamma_n^C$ remain constant.
    \item On $[t_{1,1}, t_{1,2}]$, we move robot $z_2$ from $z_2^1$ to $z_2^2$, treating $o_1, \ldots, o_m$ and $z_1^2, z_3^1, \ldots, z_n^1$ as obstacles. The paths $\gamma_1^C, \gamma_3^C, \ldots, \gamma_n^C$ remain constant.
    \item This process continues until $[t_{1,n-1}, t_{1,n}]$, where $z_n$ moves from $z_n^1$ to $z_n^2$ while treating $o_1, \ldots, o_m, z_1^2, \ldots, z_{n-1}^2$ as obstacles.
\end{itemize}
At the end of the interval $[t_1, t_2]$ the robots $z_1, \ldots z_n$ reach the states $z_1^2, \ldots, z_{n}^2$ respectively.

\medskip
\noindent
\subsection*{The paths over general interval \texorpdfstring{$[t_{k-1}, t_k]$ for $2 \leq k \leq r_n$}:}
Since the number of states $r_i$ visited by the robot $z_i$ can vary with $i \in \{1,2,\ldots,n\}$ ( $r_1 \leq r_2 \leq \cdots \leq r_n$), the above process may not directly apply to a general interval $[t_{k-1}, t_k]$,~ $2 \leq k \leq r_n$. Notice that there exists $i$ such that
$r_{i-1} < k \leq r_i, \text{ where } r_0 = 0.$
During $[0, t_{k-1}]$, robots $z_1, \ldots, z_{i-1}$ have already reached their final positions, while $z_i, \ldots, z_n$ still need to move.
We divide $[t_{k-1}, t_k]$ into $n - i + 1$ equal parts:
$$
t_{k-1} = t_{k-1,0} < t_{k-1,1} < \cdots < t_{k-1,\,n-i+1} = t_k.
$$
On $[t_{k-1,0}, t_{k-1,1}]$, we apply \cite[Section~2]{FarP23} to move the robot $z_i$ from $z_i^{k-1}$ to $z_i^k$, treating
\[
o_1, \ldots, o_m, \quad z_1^{r_1}, \ldots, z_{i-1}^{r_{i-1}}, \quad z_{i+1}^{k-1}, \ldots, z_n^{k-1}
\]
as obstacles. The remaining sub-intervals are handled similarly for robots $z_{i+1}, \ldots, z_n$.  At the end of time $t_{r_n}$ each robot reach their final state. 
This algorithm is continuous over $E_{\mu, R}^{\tau}$ and the paths $\gamma_1^C, \gamma_2^C, \ldots, \gamma_n^C$ satisfy conditions (i),(ii), and (iii).

\subsection{Algorithm over \texorpdfstring{$E_{\mu, \nu}^{\sigma}$}:} 

Let $C \in E_{\mu, \nu}^{\sigma}$ be any configuration to which we associate a positive number $\delta_C$ defined as follows. If $\mu + \nu \geq 2$, then  
$$2\delta_C:= \min \Big\{|q(z_i^{k_i})-q(z_\ell^{k_\ell})|, ~|q(z_i^{k_i})-q(o_j)|~\colon~q(z_i^{k_i})\neq q(z_\ell^{k_\ell}),~ q(z_i^{k_i})\neq q(o_j)\Big\},$$
where $i, \ell \in \{1, \ldots, n\}, ~ j\in \{1, \ldots, m\}, ~ k_i \in \{1, \ldots, r_i\}, ~ k_\ell \in \{1, \ldots, r_\ell\}$. If instead $\mu + \nu = 1$, we set $\delta_C := 1$. Notice that $\delta_C$ is continuously dependent on the configuration $C \in E_{\mu, \nu}^{\sigma}$.  Define a homotopy 
$$H:  E_{\mu, \nu}^{\sigma} \times I \to E^{\bar{r}}_B, \quad H(C, s)= \Big(o_1, \ldots, o_m, \alpha^1_{1}(s), \ldots , \alpha^{r_{1}}_{1}(s), \ldots, \alpha^1_{n}(s), \ldots , \alpha^{r_{n}}_{n}(s) \Big),$$
where, $$\alpha_i^{k_i}(s)= z_i^{k_i}+ \frac{(\Sigma_{\ell=1}^{i-1}r_\ell + k_i)\cdot s \cdot \delta_C}{R}\cdot e.$$   
One can see that $H(C, 0)=C$ and if $s>0$ then $H(C, s)\in E_{\mu, R}^{\tau}$ for some $\tau \in \Sigma_{\mu, R}$. It is easy to see that if $C'$ is any other configuration of $E_{\mu, \nu}^{\sigma}$ and $s'>0$ then $H(C', s')\in E_{\mu, R}^{\tau}.$ We define $$A_{\mu, \nu}^{\sigma}=\{ H(C, 1) ~|~ C\in E_{\mu, \nu}^{\sigma}\},$$

\noindent Then $A_{\mu, \nu}^{\sigma} \subset E_{u, R}^{\tau}$. We may apply the Algorithm \ref{Algorithm over E u,R,tau} over $A_{\mu, \nu}^{\sigma}$ and use $H$ to get a deformation between $E_{\mu, \nu}^{\sigma}$ and $A_{\mu, \nu}^{\sigma}$. As a result, we obtain an algorithm over $E_{\mu, \nu}^{\sigma}$.

This provides an algorithm over $E^{\bar{r}}_B$ that is continuous on each $W_c, ~ 1\leq c \leq R+m$.

\begin{remark}
    This shows that $\TC_{\bar{\textbf{r}}}[p\colon E \to B] \leq \displaystyle\sum^{n}_{i=1} {r_i} + m -1$ where $p$ is the fibration \eqref{Eq_FNB} and the dimension $d\geq 2$ can be even or odd. This result we already proved in the Proposition \ref{Prop_ubptc} using homotopy dimension and connectivity.
\end{remark}

\subsection*{Algorithm for the even-dimensional Euclidean space}
In this subsection, we modify the above algorithm, which will work when the dimension $d$ of the Euclidean space $\rr^d$ is even. We will see that in this case the number $\TC_{\bar{\textbf{r}}}[p\colon E \to B]$ is reduced by one.

For a configuration $C\in E^{\bar{r}}_B$ 
consider the unit vector $e_C = \frac{o_2 - o_1}{||o_2 - o_1||}\in S^{d-1}$, where $o_1$ and $o_2$ are the first two obstacles associated to $C$. Let $L_C$ be the oriented line passing through the origin in the direction of $e_C$. Since the dimension $d\geq 2$ is even so the sphere $S^{d-1}$ admits a continuous non-vanishing tangent vector field. Such a tangent vector field assigns a perpendicular unit vector $(e_C)^\perp$ corresponds to each $e_C$, continuously depending on $C\in E^{\bar{r}}_B$. Here we consider an orthogonal projection map $$q_C : \mathbb{R}^d\to L_C \text{ defined as } q_C(x)=\langle  x, e_C \rangle  \cdot  e_C.$$
and consider the set
$$q_C(C) = \{q_C(o_j), ~q_C(z_i^{k_i}): ~1\leq j \leq m,~ 1\leq i \leq n, ~ 1\leq k_i \leq r_i \}.$$
Notice that $q_C(o_1)\neq q_C(o_2)$. Therefore in this case $$2\leq |q_C(C)|\leq R + m.$$
As in the previous case here we define the set $$W_c=\big\{C\in E^{\bar{r}}_B : |q_C(C)| = c \big\},$$
where $2\leq c \leq R + m$ and we make a continuous algorithm in similar way as previously discussed in this section. Since here the number of local rule is $R+m-1$, the topological complexity $\TC_{\bar{\textbf{r}}}[p\colon E \to B] \leq \displaystyle\sum^{n}_{i=1} {r_i} + m -2$. Using the Proposition \ref{proposition lower bound even case}, we get the following theorem.

\begin{theorem}\label{Th_even case}
The $\bar{\textbf{r}}$-sequential parametrized topological complexity of the Fadell-Neuwirth fibration \eqref{Eq_FNB} is given by $$\TC_{\br}[p\colon F(\mathbb{R}^d, n+ m) \to F(\mathbb{R}^d, m)] =\sum^{n}_{i=1} {r_i} + m -2,$$
where $d\geq 2$ is even, $n\geq 1$, and $m\geq 2$. 
\end{theorem}

\sect{Appendix}\label{sec:appendix}

\subsection{Appendix I-1}\label{Appendix for odd case}
In the proof of Proposition \ref{Prop_odd lower bound}, we skipped the laborious calculations to maintain the flow of the argument. Here, we provide a brief outline of these calculations. Using Proposition \ref{Prop_relations on cohomology classes}(a), we obtain the following simplification: 
\begin{align*}
x^0&=\prod^{m+n_1}_{j=m+1}(w^2_{1j} - w^1_{1j}) \prod^{r_{n_1}}_{s=2} ~\prod^{m+n}_{j=m+1} (w^s_{1j} - w^1_{1j}) \\
&= \prod^{m+n_1}_{j=m+1}(w^2_{1j} - w^1_{1j}) \prod^{m+n}_{j=m+1} (w^2_{1j} - w^1_{1j}) \prod^{r_{n_1}}_{s=3} ~\prod^{m+n}_{j=m+1} (w^s_{1j} - w^1_{1j})\\
&= \prod^{m+n_1}_{j=m+1}(w^2_{1j} - w^1_{1j})^2 \prod^{m+n}_{j=m+n_1+1}(w^2_{1j} - w^1_{1j}) \prod^{r_{n_1}}_{s=3} ~\prod^{m+n}_{j=m+1} (w^s_{1j} - w^1_{1j})\\
&= (-2)^{n_1}\prod^{m+n_1}_{j=m+1} w^1_{1j}w^2_{1j} \prod^{m+n}_{j=m+n_1+1}(w^2_{1j} - w^1_{1j}) ~\prod^{m+n}_{j=m+1} (w^3_{1j} - w^1_{1j})\cdots (w^{r_{n_1}}_{1j} - w^1_{1j})\\
&= (-2)^{n_1}\prod^{m+n_1}_{j=m+1} w^1_{1j}w^2_{1j} \prod^{m+n}_{j=m+n_1+1}(w^2_{1j} - w^1_{1j}) \\
& \hspace{4.5cm} \prod^{m+n}_{j=m+1} \Big(w^3_{1j} \cdots w^{r_{n_1}}_{1j} - w^1_{1j} \sum^{r_{n_1}}_{a=3}  w^3_{1j} \cdots \hat{w}^a_{1j}\cdots w^{r_{n_1}}_{1j} \Big)\\
&=(-2)^{n_1} \prod^{m+n_1}_{j=m+1} w^1_{1j}w^2_{1j}  \prod^{m+n_1}_{j=m+1} \Big(w^3_{1j} \cdots w^{r_{n_1}}_{1j} - w^1_{1j} \sum^{r_{n_1}}_{a=3}  w^3_{1j} \cdots \hat{w}^a_{1j}\cdots w^{r_{n_1}}_{1j} \Big)\\ 
&\hspace{3cm} \prod^{m+n}_{j=m+n_1+1} (w^2_{1j} - w^1_{1j}) \Big(w^3_{1j} \cdots w^{r_{n_1}}_{1j} - w^1_{1j} \sum^{r_{n_1}}_{a=3}  w^3_{1j} \cdots \hat{w}^a_{1j}\cdots w^{r_{n_1}}_{1j} \Big)  \\
&=(-2)^{n_1}\prod^{m+n_1}_{j=m+1} w^1_{1j} \cdots w^{r_{n_1}}_{1j}  \prod^{m+n}_{j=m+n_1+1} \Big(w^2_{1j} \cdots w^{r_{n_1}}_{1j} - w^1_{1j} \sum^{r_{n_1}}_{a=2}  w^2_{1j} \cdots \hat{w}^a_{1j}\cdots w^{r_{n_1}}_{1j} \Big)
\end{align*}
where $\hat{w}^a_{1j}$ denotes that the term ${w}^a_{1j}$ is absent in the product inside the summation. 
With similar simplification, we have
$$x^{\ell-1} = (-2)^{(n-n_{\ell-1})} \prod^{m+n}_{j=m+n_{\ell-1}+1} w^1_{1j} w^{r_{n_{\ell-1}}+1}_{1j} \cdots w^{r_n}_{1j}.$$
For $u = 1,\ldots, \ell-2$, we get
\begin{align*}
x^u =  (-2)^{(n_{u+1}- n_u)} & \prod^{m+n_{u+1}}_{j=m + n_u + 1} w^1_{1j} w^{r_{n_u}+1}_{1j} \cdots w^{r_{n_{u+1}}}_{1j}\\
 & \prod^{m+n}_{j=m+n_{u+1}+1} \Big(w^{r_{n_u}+1}_{1j} \cdots w^{r_{n_{u+1}}}_{1j} - w^1_{1j} \sum^{r_{n_{u+1}}}_{a=r_{n_u}+1}  w^{r_{n_u}+1}_{1j} \cdots \hat{w}^a_{1j}\cdots w^{r_{n_{u+1}}}_{1j}\Big).
\end{align*}
Now consider the product
\begin{multline*}
x^0 x^1= (-2)^{n_2} \prod^{m+n_1}_{j=m+1} w^1_{1j}\cdots w^{r_{n_1}}_{1j} \prod^{m+n_2}_{j=m+n_1 +1} w^1_{1j}\cdots w^{r_{n_2}}_{1j} 
\\
\prod^{m+n}_{j=m+n_2+1}\Big({w}^2_{1j}\cdots w^{r_{n_2}}_{1j} - w^1_{1j} \sum^{r_{n_2}}_{a=2} w^2_{1j}\cdots \hat{w}^a_{1j}\cdots w^{r_{n_2}}_{1j} \Big).
\end{multline*}
Similarly,
\begin{align*}
x^0 x^1 x^2 = & ~(-2)^{n_3} \prod^{m+n_1}_{j=m+1} w^1_{1j}\cdots w^{r_{n_1}}_{1j} \prod^{m+n_2}_{j=m+n_1 +1} w^1_{1j}\cdots w^{r_{n_2}}_{1j} \prod^{m+n_3}_{j=m+n_2+1} w^1_{1j}\cdots w^{r_{n_3}}_{1j}\\
& \hspace{2cm} \prod^{m+n}_{j=m+n_3+1} \Big(w^2_{1j}\cdots w^{r_{n_3}}_{1j} - w^1_{1j} \sum^{r_{n_3}}_{a=2} w^2_{1j}\cdots \hat{w}^a_{1j}\cdots w^{r_{n_3}}_{1j}\Big).
\end{align*}

\pagebreak

\noindent Continuing this multiplication it follows that 
\begin{align*}
x^0 x^1 \cdots x^{\ell-2} = & (-2)^{n_{\ell-1}} \prod^{m+n_1}_{j=m+1} w^1_{1j}\cdots w^{r_{n_1}}_{1j} \prod^{m+n_2}_{j=m+n_1 +1} w^1_{1j}\cdots w^{r_{n_2}}_{1j}\cdots \prod^{m+n_{\ell-1}}_{j=m+n_{\ell-2}+1} w^1_{1j}\cdots w^{r_{n_{\ell-1}}}_{1j}\\
& \hspace{2cm} \prod^{m+n}_{j=m+n_{\ell-1}+1} \Big(w^2_{1j}\cdots w^{r_{n_{\ell-1}}}_{1j} - w^1_{1j} \sum^{r_{n_{\ell-1}}}_{a=2} w^2_{1j}\cdots \hat{w}^a_{1j}\cdots w^{r_{n_{\ell-1}}}_{1j}\Big).
\end{align*}
Therefore, 
\begin{align*}
x^0 x^1 \cdots x^{\ell-2} x^{\ell-1} = & (-2)^n \prod^{m+n_1}_{j=m+1} w^1_{1j}\cdots w^{r_{n_1}}_{1j} \cdots \prod^{m+n_{\ell-1}}_{j=m+n_{\ell-2}+1} w^1_{1j}\cdots w^{r_{n_{\ell-1}}}_{1j} \prod^{m+n}_{j=m+n_{\ell-1}+1} w^1_{1j}\cdots w^{r_n}_{1j}\\
= &(-2)^n \prod^{m+n}_{j=m+1} w^1_{1j}\cdots w^{r_{n_1}}_{1j} \prod^{m+n}_{j=m+n_1+1} w^{r_{n_1}+1}_{1j} \cdots w^{r_{n_2}}_{1j} \cdots \prod^{m+n}_{j=m+n_{\ell-1}+1} w^{r_{\ell-1}+1}_{1j}\cdots w^{r_n}_{1j}.
\end{align*}

\subsection{Appendix I-2}\label{Appendix odd only once}
It is clear that $I'=\emptyset$ whenever $I=\{2,\ldots,m\}$. Using Proposition \ref{Prop_relations on cohomology classes}(b), we have
\begin{align*}
w^1_{IJ}w^2_{I'J'}w^1_{1(m+1)}w^2_{1(m+1)} = & ~w^1_{IJ}w^1_{1(m+1)}w^2_{1(m+1)}~~(\text{ since } w^2_{I'J'}=1)\\
= & ~w^1_{2(m+1)}w^1_{3(m+1)}\cdots w^1_{m(m+1)}w^1_{1(m+1)}w^2_{1(m+1)} \\
= & ~w^1_{1(m+1)}w^1_{2(m+1)}w^1_{3(m+1)}\cdots w^1_{m(m+1)}w^2_{1(m+1)} \\
= & ~w_{12}(w^1_{2(m+1)}-w^1_{1(m+1)})w^1_{3(m+1)}\cdots w^1_{m(m+1)}w^2_{1(m+1)} \\
= & ~w_{12}w^1_{2(m+1)}w^1_{3(m+1)}\cdots w^1_{m(m+1)}w^2_{1(m+1)}-  \\
& ~~ w_{12}w^1_{1(m+1)}w^1_{3(m+1)}\cdots w^1_{m(m+1)}w^2_{1(m+1)} \\
= & ~w_{12}w_{23}(w^1_{3(m+1)} - w^1_{2(m+1)})\cdots w^1_{m(m+1)}w^2_{1(m+1)} - \\
 & ~~ w_{12}w_{13}(w^1_{3(m+1)}-w^1_{1(m+1)})\cdots w^1_{m(m+1)}w^2_{1(m+1)} 
\end{align*}
and so on. Following this process, we finally obtain that the term $$w_{12}w_{23}\cdots w_{2m}w^1_{2(m+1)}w^2_{1(m+1)}$$ appears in $w^1_{IJ}w^2_{I'J'}w^1_{1(m+1)}w^2_{1(m+1)}$ exactly once. 
Similarly, $I=\emptyset$ whenever $I' = \{2,\ldots,m\}$, and we can get that the term $w_{12}w_{23}\cdots w_{2m}w^2_{2(m+1)}w^1_{1(m+1)}$ appears in $w^1_{IJ}w^2_{I'J'}w^1_{1(m+1)}w^2_{1(m+1)}$ that is close to our previous term $w_{12}w_{23}\cdots w_{2m}w^1_{2(m+1)}w^2_{1(m+1)}$ but not equal.

\subsection{Appendix II} \label{app_proposition lower bound even case}
Here, we present the calculation that was skipped in Proposition \ref{proposition lower bound even case} for the sake of maintaining the fluency of the argument. Consider
\begin{equation}\label{Eq_zj}
z_j = \Big(w^2_{1(j-1)} w^2_{(j-1)j} - w^2_{1(j-1)} w^2_{1j} - w^1_{1j} w^2_{(j-1)j} +w^1_{(j-1)j} w^2_{1j} + w^1_{1(j-1)} w^1_{(j-1)j} - w^1_{1(j-1)} w^1_{1j}\Big),
\end{equation}
and hence $\tilde{y}' = \displaystyle\prod^{m+n}_{j=m+2} z_j$. Thus, we have 

\pagebreak

\begin{align*}
\tilde{y}' \tilde{y}^0 y^1 \cdots y^{\ell-1} 
= & \pm \big(w^3_{1(m+1)} \cdots w^{r_{n_1}}_{1(m+1)} \pm w^1_{1(m+1)} \sum^{r_{n_1}}_{a=3} w^3_{1(m+1)} \cdots \hat{w}^a_{1(m+1)} \cdots w^{r_{n_1}}_{1(m+1)}\big)\\
& \hspace{1cm} \prod^{m+n_1}_{j=m+2} z_j \Big(w^3_{1j} \cdots w^{r_{n_1}}_{1j} \pm w^1_{1j} \sum^{r_{n_1}}_{a=3} w^3_{1j} \cdots \hat{w}^a_{1j} \cdots w^{r_{n_1}}_{1j}\Big)\\
& \hspace{1.5cm} \prod^{m+n_2}_{j=m+n_1 +1} z_j \Big(w^3_{1j} \cdots w^{r_{n_2}}_{1j} \pm w^1_{1j} \sum^{r_{n_2}}_{a=3} w^3_{1j} \cdots \hat{w}^a_{1j} \cdots w^{r_{n_2}}_{1j}\Big) \cdots \\
& \hspace{2cm} \prod^{m+n}_{j=m+n_{\ell-1} +1} z_j \Big(w^3_{1j} \cdots w^{r_{n}}_{1j} \pm w^1_{1j} \sum^{r_{n}}_{a=3} w^3_{1j} \cdots \hat{w}^a_{1j} \cdots w^{r_{n}}_{1j}\Big).
\end{align*}
Now we want to understand the product. One may think of the six summands of $z_j$ in \eqref{Eq_zj} as $x_j^i$ for $i=1, \dots, 6$; i.e. $x_j^1=w^2_{1(j-1)}w^2_{(j-1)j}$, $x_j^2=w^2_{1(j-1)}w^2_{1j}$ and so on. For $u=1,\ldots,\ell$, we assume $\big(w^3_{1j} \cdots w^{r_{n_u}}_{1j} \pm w^1_{1j} \displaystyle\sum^{r_{n_u}}_{a=3} w^3_{1j} \cdots \hat{w}^a_{1j} \cdots w^{r_{n_u}}_{1j}\big)$ by $W^u_j=W^{1u}_j \pm W^{2u}_j$. 
Notice that the product $x^3_j W^{1u}_j$ produces $w^1_{1j}w^2_{(j-1)j}w^3_{1j}\cdots w^{r_{n_u}}_{1j}$ that neither of the other products produces nor cancels. Similarly, $x^4_j W^{1u}_j$ yields $w^1_{(j-1)j}w^2_{1j}w^3_{1j} \cdots w^{r_{n_u}}_{1j}$ that neither of the other products produces nor cancels.
Thus, for $u=1,\ldots, \ell$, each factor of the form $z_j \Big(w^3_{1j} \cdots w^{r_{n_u}}_{1j} \pm w^1_{1j} \displaystyle\sum^{r_{n_i}}_{a=3} w^3_{1j} \cdots \hat{w}^a_{1j} \cdots w^{r_{n_u}}_{1j}\Big)$ contains the following two terms with a single occurrence each:
$$w^1_{(j-1)j}w^2_{1j}w^3_{1j} \cdots w^{r_{n_u}}_{1j} \quad \text{ and } \quad w^1_{1j}w^2_{(j-1)j}w^3_{1j}\cdots w^{r_{n_u}}_{1j}.$$ 
It follows that the product $ \displaystyle\prod^{m+n_1}_{j=m+2} z_j \Big(w^3_{1j} \cdots w^{r_{n_1}}_{1j} \pm w^1_{1j} \sum^{r_{n_1}}_{a=3} w^3_{1j} \cdots \hat{w}^a_{1j} \cdots w^{r_{n_1}}_{1j}\Big)$ contains the term $\displaystyle\prod^{m+n_1}_{j=m+2}w^1_{1j}w^2_{(j-1)j}w^3_{1j} \cdots w^{r_{n_1}}_{1j}$ only once. For $u= 2,\ldots, \ell$, we use similar argument to conclude that the product $ \displaystyle\prod^{m+n_u}_{j=m+n_{u-1}+1} z_j \Big(w^3_{1j} \cdots w^{r_{n_u}}_{1j} \pm w^1_{1j} \sum^{r_{n_u}}_{a=3} w^3_{1j} \cdots \hat{w}^a_{1j} \cdots w^{r_{n_u}}_{1j}\Big)$ contains the term $\displaystyle\prod^{m+n_u}_{j=m+n_{u-1}+1}w^1_{1j}w^2_{(j-1)j}w^3_{1j} \cdots w^{r_{n_u}}_{1j}$ precisely once. Hence, their product occur only once in $\tilde{y}' \tilde{y}^0 y^1 \cdots y^{\ell-1}$.

\vspace{1cm}

\noindent {\bf Acknowledgment.}
The first and second authors thank NBHM for their postdoctoral fellowships. The third author thanks IIT Kanpur for its postdoctoral fellowship.

\bibliographystyle{abbrv}
\bibliography{TC-reference}

\end{document}